\documentclass[12pt,reqno]{amsart} 

\usepackage{mathrsfs,amsfonts,amsthm,amsmath,amssymb} 
\usepackage[dvipsnames]{xcolor} 

\usepackage{enumitem}                
\usepackage{lineno}                 

\usepackage[numbers,sort&compress]{natbib}  

\usepackage{graphics}
\usepackage{graphicx}   
\graphicspath{{images/}} 
\usepackage{subfig}     
\usepackage{rotating}

\usepackage[hyperindex, pdfencoding = auto, psdextra, bookmarksdepth = 4]{hyperref}  
\hypersetup{
    bookmarksopen      = true,
    bookmarksopenlevel = 1,
    bookmarksnumbered  = true,
    pdftitle    = {},
    pdfauthor   = {},
    linktoc     = page,
    anchorcolor = green,
    colorlinks  = true,        
    linkcolor   = black,
    citecolor   = blue,
    filecolor   = blue,
    urlcolor    = magenta,
}

\usepackage[capitalise, noabbrev, nameinlink]{cleveref}    

\theoremstyle{plain}

\newtheorem{theorem}{Theorem}[section]
\newtheorem{lemma}[theorem]{Lemma}

\newtheorem{proposition}[theorem]{Proposition}
\newtheorem{corollary}[theorem]{Corollary}
\newtheorem{question}[theorem]{Question}

\theoremstyle{definition}   
\newtheorem{definition}[theorem]{Definition}
\newtheorem{remark}[theorem]{Remark}
\newtheorem{example}[theorem]{Example}

\makeatletter
    \def\subsection{\@startsection{subsection}{2}%
    \z@{.5\linespacing\@plus.7\linespacing}{.3\linespacing}%
    {\normalfont\bfseries}}
\makeatother 

\makeatletter                                         
    \newcommand{\LeftEqNo}{\let\veqno\@@leqno}        
\makeatother

\numberwithin{equation}{section}                      

\begin{document}

\

\vspace{-2cm}

\title[The similarity of irreducible operators in factors]{The similarity of irreducible operators in factors}

\author{Minghui Ma}
\address{Minghui Ma, School of Mathematical Sciences, Dalian University of Technology, Dalian, 116024, China}
\email{minghuima@dlut.edu.cn}


\author{Rui Shi}
\address{Rui Shi, School of Mathematical Sciences, Dalian University of Technology, Dalian, 116024, China}
\email{ruishi@dlut.edu.cn}
\thanks{Minghui Ma was supported by the China Postdoctoral Science Foundation (No.2025M783077) and the Postdoctoral Fellowship Program of CPSF (No.GZC20252022). Rui Shi was supported by NSFC (No.12271074). 
}

\author{Shanshan Yang}
\address{Shanshan Yang, School of Mathematical Sciences, Dalian University of Technology, Dalian, 116024, China}
\email{yss@mail.dlut.edu.cn}


\begin{abstract}
    An operator $T$ in a separable factor $\mathcal{M}$ is said to be irreducible in $\mathcal{M}$ if the von Neumann subalgebra $W^*(T)$ generated by $T$ is an irreducible subfactor of $\mathcal{M}$, i.e., $W^*(T)'\cap\mathcal{M}=\mathbb{C}I$.
    We say that $T$ is a single generator of $\mathcal{M}$ if $W^*(T)=\mathcal{M}$.
    In this paper, we study generators of separable factors related to maximal abelian
    self-adjoint subalgebras.
    As an application, we obtain a complete characterization of normal operators in separable factors which are similar to irreducible operators.
\end{abstract}

\maketitle

\section{Introduction}

Let $\mathcal{H}$ be a separable complex Hilbert space and $\mathcal{B}(\mathcal{H})$ the algebra of all bounded linear operators on $\mathcal{H}$.
An operator $T$ in $\mathcal{B}(\mathcal{H})$ is said to be {\em irreducible} if there is no nontrivial projection commuting with $T$.
Otherwise, $T$ is called {\em reducible}.
In 1968, P. Halmos \cite{Hal68} introduced the concept of irreducibility of operators and proved that the set of irreducible operators is a norm-dense $G_\delta$ subset of $\mathcal{B}(\mathcal{H})$.
The irreducibility of operators is clearly a unitary invariant, and it is natural to seek necessary and sufficient conditions for an operator to be similar to an irreducible operator.
Along this line, F. Gilfeather \cite[Theorem 2]{Gil72} proved that, on a separable Hilbert space, every normal operator with empty point spectrum is similar to an irreducible operator.
Subsequent to Gilfeather's result, C. Fong and C. Jiang \cite[Main theorem]{FJ94} gave a complete characterization of normal operators in $\mathcal{B}(\mathcal{H})$ that are similar to irreducible operators.
The similarity of irreducible operators in matrix algebras was completely characterized by C. Hsin \cite{Hsin00}.
The reader is referred to \cite[Section 2.2]{JW98-book} for more results on the similarity of irreducible operators.
Noting that $\mathcal{B}(\mathcal{H})$ is a type $\mathrm{I}$ factor, and some related operator-theoretic results \cite{CFY21,CFY24,DK21,DSSW,DSZ15,HS09,HMS26,MSSW,Shen09,SS19} have been established in factors recently, it is natural to investigate the similarity of irreducible operators in factors.
As the proof of \cite[Main theorem]{FJ94} makes essential use of specific features of $\mathcal{B}(\mathcal{H})$, obtaining an analog in general von Neumann algebras requires the development of new methods.

Let us recall some basic terminology from the theory of von Neumann algebras.
A self-adjoint subalgebra of $\mathcal{B}(\mathcal{H})$ is said to be a {\em von Neumann algebra} if it is closed in the weak-operator topology and contains $I$, the identity operator of $\mathcal{B}(\mathcal{H})$.
A \emph{factor} is a von Neumann algebra with trivial center.
By the type decomposition theorem \cite[Theorem 6.5.2]{KR2}, factors are classified into type $\mathrm{I}_n$, type $\mathrm{I}_\infty$, type $\mathrm{II}_1$, type $\mathrm{II}_\infty$, and type $\mathrm{III}$.
Let $\mathcal{M}$ be a separable von Neumann algebra, i.e., it has a separable predual.
For any subset $\mathcal{S}$ of $\mathcal{M}$, we denote by $W^*(\mathcal{S})$ the smallest von Neumann subalgebra containing $\mathcal{S}$.
An operator $T$ in $\mathcal{M}$ is said to be a \emph{single generator} of $\mathcal{M}$ if $\mathcal{M}=W^*(T)$.
The generator problem, raised by R. Kadison in the 1960s, asks whether each separable von Neumann algebra is singly generated \cite{Ge03}.
Although significant progress has been made over several decades, the problem remains open for type $\mathrm{II}_1$ factors.
Note that the answer is affirmative for separable von Neumann algebras without a direct summand of type $\mathrm{II}_1$ (see \cite{SS63,Wog69}).
To provide a unified approach to the generator problem for type $\mathrm{II}_1$ factors, J. Shen \cite{Shen09} introduced his invariant $\mathcal{G}$ to measure the number of generators and proved that $\mathcal{G}(\mathcal{M})=0$ for type $\mathrm{II}_1$ factors with Cartan subalgebras, those with property $\Gamma$, nonprime factors, or some type $\mathrm{II}_1$ factors with property $(T)$.
Later, this invariant $\mathcal{G}$ was further studied in \cite{DSSW}.
In particular, it follows from \cite[Theorem 3.1]{DSSW} that $\mathcal{M}$ is singly generated whenever $\mathcal{G}(\mathcal{M})<\frac{1}{2}$.
The reader is referred to \cite[Chapter 16]{SS-book} for more details on the generator problem.
To establish the main result \Cref{thm similarity} of this paper, we require the following theorem concerning maximal abelian self-adjoint subalgebras, which is of independent interest.

\begin{theorem}\label{thm masa}
    Let $\mathcal{M}$ be a separable factor and $\mathcal{A}$ a maximal abelian self-adjoint subalgebra of $\mathcal{M}$.
    Suppose that one of the following conditions holds:
    \begin{enumerate} [label= $(\arabic*)$, ref= \arabic*, leftmargin=*] 
        \item[$(1)$] $\mathcal{M}$ is of type $\mathrm{I}_n$ or type $\mathrm{III}$;
        \item[$(2)$] $\mathcal{M}$ is of type $\mathrm{II}_1$, and $\mathcal{G}(\mathcal{M})=0$;
        \item[$(3)$] $\mathcal{M}$ is of type $\mathrm{I}_\infty$ or type $\mathrm{II}_\infty$, and
        \begin{equation*}
            I=\bigvee\{Q\in\mathcal{A}\colon Q~\text{is a finite projection in}~\mathcal{M}\}.
        \end{equation*}
    \end{enumerate}
    Then for every nonzero projection $P$ in $\mathcal{A}$ with $P\precsim I-P$ in $\mathcal{M}$, there exists a unitary operator $U$ in $\mathcal{M}$ such that $\mathcal{M}=W^*(U^*PU,\mathcal{A}(I-P))$.
\end{theorem}

Let $\mathcal{M}$ be a separable factor and we further assume that $\mathcal{G}(\mathcal{M})=0$ if $\mathcal{M}$ is of type $\mathrm{II}_1$.
As an application of \Cref{thm masa}, we prove in \Cref{prop FT} that a self-adjoint operator in $\mathcal{M}$ is the real part of a generator if
and only if it is not a scalar multiple of the identity.
This result can be viewed as a generalization of a result of P. Fillmore and D. Topping \cite[Theorem 2]{FT69}.

Two projections $P$ and $Q$ in $\mathcal{M}$ are said to be {(\emph{Murray-von Neumann}) \emph{equivalent}}, denoted by $P\sim Q$, if $V^*V=P$ and $VV^*=Q$ for some partial isometry $V$ in $\mathcal{M}$.
We say that $P$ is \emph{weaker than $Q$}, denoted by $P\precsim Q$ or $Q\succsim P$, if $P\sim Q_0\leqslant Q$ for some projection $Q_0$ in $\mathcal{M}$.
We write $P\prec Q$ if $P$ is strictly weaker than $Q$, i.e., $P$ is weaker than $Q$ and is not equivalent to $Q$.
Let us briefly explain the condition $P\succsim I-P$ appearing in \Cref{thm similarity}, where $P$ is a projection in $\mathcal{M}$.
In the matrix algebra $M_n(\mathbb{C})$, the condition $P\succsim I-P$ is equivalent to $\operatorname{rank}(P)\geqslant\frac{n}{2}$.
Similarly, in a type $\mathrm{II}_1$ factor $(\mathcal{M},\tau)$ with a unique normal faithful tracial state $\tau$, the condition $P\succsim I-P$ is equivalent to $\tau(P)\geqslant\frac{1}{2}$.
While in properly infinite factors, i.e., factors of type $\mathrm{I}_\infty$, type $\mathrm{II}_\infty$, or type $\mathrm{III}$, the condition $P\succsim I-P$ states that $P$ is an infinite projection.
It is worth mentioning that every nonzero projection in type $\mathrm{III}$ factors is infinite.

Recall that an operator $T$ in a separable factor $\mathcal{M}$ is said to be \emph{irreducible} if the von Neumann algebra $W^*(T)$ generated by $T$ is an irreducible subfactor of $\mathcal{M}$, i.e., $W^*(T)'\cap\mathcal{M}=\mathbb{C}I$.
Otherwise, $T$ is called \emph{reducible} in $\mathcal{M}$.
By definition, every scalar in $\mathbb{C}$ is irreducible.
Moreover, it is straightforward to prove that an operator in $M_2(\mathbb{C})$ is similar to an irreducible operator if and only if it is not a scalar multiple of the identity.
We present the main result of this paper as follows, where by $R(T)$ we denote the range projection of $T$.

\begin{theorem}\label{thm similarity}
    Let $\mathcal{M}$ be a separable factor with $\dim\mathcal{M}>4$.
    Then a normal operator $N$ in $\mathcal{M}$ is similar to an irreducible operator if and only if the following two conditions are satisfied:
    \begin{enumerate}  [label= $(\arabic*)$, ref= \arabic*, leftmargin=*] 
        \item \label{item-a1} $R(\lambda I-N)\succsim I-R(\lambda I-N)$ for every scalar $\lambda$ in $\mathbb{C}$;
        \item \label{item-a2} the set $\{I,N,N^2\}$ is linearly independent over $\mathbb{C}$.
    \end{enumerate}
\end{theorem}

In 1954, N. Dunford \cite{Dun54} introduced the concept of spectral operators on a Banach space and proved that every spectral operator on a Hilbert space is similar to the sum of a normal operator and a commuting quasinilpotent operator.
By replacing $\mathcal{B}(\mathcal{H})$ with a factor $\mathcal{M}$, K. Dykema and A. Krishnaswamy-Usha \cite[Section 3]{DK21} showed that if $T$ is a spectral operator in $\mathcal{M}$, then there exists an invertible operator $X$ in $\mathcal{M}$ such that $XTX^{-1}$ is the sum of a normal operator and a commuting quasinilpotent operator.
Combining this with \Cref{thm similarity}, we obtain \Cref{cor spectral-operator} for spectral operators.
In addition, we prove in \Cref{prop commutant-abelian} that if the relative commutant $\{T\}'\cap\mathcal{M}$ is abelian, then $T$ is similar to an irreducible operator in $\mathcal{M}$.
We conclude the introduction with the following question.

\begin{question}\label{que similarity}
    In a separable factor, what are the necessary and sufficient conditions for an operator to be similar to an irreducible operator?
    What are the necessary and sufficient conditions for an operator to be similar to a generator?
\end{question}

By applying Haagerup-Schultz projections \cite{HS09}, the authors in \cite{DSZ15} proved that every operator in a finite von Neumann algebra can be written as the sum of a normal operator and an s.o.t.-quasinilpotent operator.
Recently, K. Dykema and A. Krishnaswamy-Usha \cite{DK21} investigated the class of operators in a finite von Neumann algebra that are similar to the sum of a normal operator and a commuting s.o.t.-quasinilpotent operator.
This class of operators is larger than Dunford's notion of spectral operators.
It is therefore natural to investigate whether a result analogous to \Cref{cor spectral-operator} holds for this class of operators in type $\mathrm{II}_1$ factors.

This paper is organized as follows.
In \Cref{sec perliminary}, we present some necessary techniques concerning projections.
After proving the key tool \Cref{prop U*PU-Pj}, we provide the proof of \Cref{thm masa} in \Cref{sec generator}.
As an application, \Cref{cor FT} states that a self-adjoint operator in a separable factor is the real part of an irreducible operator if and only if it is not a scalar multiple of the identity.
\Cref{sec similarity} is devoted to the proof of \Cref{thm similarity}.
In \Cref{lem main-pre}, we give a sufficient condition for an operator in a separable factor to be strongly reducible.
Finally, some related results are presented in \Cref{cor spectral-operator} and \Cref{prop commutant-abelian}.

\section{Preliminaries}\label{sec perliminary}

Throughout this paper, we assume that $\mathcal{M}$ is a separable factor, i.e., $\mathcal{M}$ has trivial center and separable predual.
For every subset $\mathcal{S}$ of $\mathcal{M}$, in contrast to $W^*(\mathcal{S})$, the von Neumann subalgebra generated by $\mathcal{S}$, we introduce the following notation.

\begin{definition}\label{def W0}
    For every subset $\mathcal{S}$ of $\mathcal{M}$, we denote by $W_0^*(\mathcal{S})$ the minimal weak-operator closed self-adjoint subalgebra containing $\mathcal{S}$. Note that $W_0^*(\mathcal{S})$ might not contain the identity operator $I$ of $\mathcal{M}$. 
\end{definition}

We denote by $R(T)$ the range projection of an operator $T$ in $\mathcal{M}$.
For every subset $\mathcal{S}$ of $\mathcal{M}$, it follows from \cite[Proposition 5.1.8]{KR1} that the identity operator of $W_0^*(\mathcal{S})$ is given by
\begin{equation} \label{eq-proj-P-Wstar0(S)}
    P:= \bigvee\{R(T)\colon T\in\mathcal{S}\cup\mathcal{S}^*\}. 
\end{equation}
By definition, it is clear that $W^*(\mathcal{S}) = W_0^*(\mathcal{S}\cup\{I\})$, and the projection $P$ as defined in \eqref{eq-proj-P-Wstar0(S)} belongs to the center of $W^*(\mathcal{S})$.
Consequently, if $\mathcal{M}$ is a factor such that $\dim\mathcal{M}>1$ and $\mathcal{M}=W^*(\mathcal{S})$, then $\mathcal{M}=W_0^*(\mathcal{S})$.
We provide a routine construction for generators of von Neumann algebras in the following lemma.
This lemma, together with \Cref{cor M4}, is prepared for the proof of \Cref{lem 1=t1+t2}.

\begin{lemma}\label{lem M4}
Let $\mathcal{N}$ be a separable factor generated by three self-adjoint operators.
Then there are projections $P_1$, $P_2$, and $P_3$ in $M_4(\mathbb{C})\otimes\mathcal{N}$ such that
\begin{equation*}
  M_4(\mathbb{C})\otimes\mathcal{N}=W^*(P_1,P_2,P_3)\quad\text{and}\quad P_2P_3=0.
\end{equation*}
\end{lemma}

\begin{proof}
Let $\mathcal{N}_1=M_2(\mathbb{C})\otimes\mathcal{N}$ and $\mathcal{N}_2=M_2(\mathbb{C})\otimes\mathcal{N}_1$.
By assumption, we may assume that $\mathcal{N}=W^*(A,B,C)$ and $\frac{1}{2}I_{\mathcal{N}}\leqslant A,B,C\leqslant\frac{3}{4}I_{\mathcal{N}}$.
Let
\begin{equation*}
  Q_1=
  \begin{pmatrix}
    A & \sqrt{A-A^2} \\
    \sqrt{A-A^2} & I_{\mathcal{N}}-A
  \end{pmatrix}\quad\text{and}\quad
  B_1=
  \begin{pmatrix}
    B & 0 \\
    0 & \frac{1}{2}C
  \end{pmatrix}.
\end{equation*}
It is clear that $\mathcal{N}_1=W^*(Q_1,B_1)$, where $Q_1$ is a projection and $\frac{1}{4}I_{\mathcal{N}_1}\leqslant B_1\leqslant\frac{3}{4}I_{\mathcal{N}_1}$.
We construct three projections in $\mathcal{N}_2$ as follows
\begin{equation*}
  P_1=
  \begin{pmatrix}
    B_1 & \sqrt{B_1-B_1^2} \\
    \sqrt{B_1-B_1^2} & I_{\mathcal{N}_1}-B_1
  \end{pmatrix},\quad
  P_2=
  \begin{pmatrix}
    Q_1 & 0 \\
    0 & 0
  \end{pmatrix},\quad
  P_3=
  \begin{pmatrix}
    I_{\mathcal{N}_1}-Q_1 & 0 \\
    0 & 0
  \end{pmatrix}.
\end{equation*}
Then $\mathcal{N}_2=W^*(P_1,P_2,P_3)$ and $P_2P_3=0$.
This completes the proof.
\end{proof}

Let $(\mathcal{M},\tau)$ be a separable type $\mathrm{II}_1$ factor.
Suppose that $t$ is a positive scalar, $n$ is a positive integer with $n\geqslant t$, and $P$ is a projection in $\mathcal{M}$ with $\tau(P)=\frac{t}{n}$.
Then we define $\mathcal{M}_t=M_n(\mathbb{C})\otimes P\mathcal{M}P$.
Up to isomorphism, $\mathcal{M}_t$ is independent of the choice of $n$ and $P$.
Recall that $\mathcal{G}$ is Shen's invariant for separable type $\mathrm{II}_1$ factors \cite{Shen09}.
By \cite[Theorem 4.5]{DSSW}, we have the following scaling formula
\begin{equation}\label{equ SS-16.4.5}
  \mathcal{G}(\mathcal{M}_t)=t^{-2}\mathcal{G}(\mathcal{M})\quad\text{for all}~t>0.
\end{equation}
The following result is an immediate consequence of \Cref{lem M4}.

\begin{corollary}\label{cor M4}
Suppose that $\mathcal{M}$ is a separable type $\mathrm{II}_1$ factor with $\mathcal{G}(\mathcal{M})<\frac{1}{16}$.
Then there are projections $P_1$, $P_2$, and $P_3$ in $\mathcal{M}$ such that
\begin{equation*}
  \mathcal{M}=W^*(P_1,P_2,P_3)\quad\text{and}\quad P_2P_3=0.
\end{equation*}
\end{corollary}

\begin{proof}
Let $\mathcal{N}$ be a separable type $\mathrm{II}_1$ factor such that $\mathcal{M}\cong M_4(\mathbb{C})\otimes\mathcal{N}$.
It follows from \eqref{equ SS-16.4.5} that
\begin{equation*}
  \mathcal{G}(\mathcal{N})=\mathcal{G}(\mathcal{M}_{1/4})
  =16\mathcal{G}(\mathcal{M})<1.
\end{equation*}
Hence $\mathcal{N}$ is generated by three self-adjoint operators by \cite[Theorem 16.6.6]{SS-book}.
This completes the proof by \Cref{lem M4}.
\end{proof}

Next, we shall prepare some basic properties of maximal abelian self-adjoint subalgebras in separable factors.
As usual, we denote by $\mathrm{Tr}$ the standard trace on the matrix algebra $M_n(\mathbb{C})$.
Recall that the diagonal algebra $D_n(\mathbb{C})$ is the set of all diagonal operators in $M_n(\mathbb{C})$ and every maximal abelian self-adjoint subalgebra $\mathcal{A}$ of $M_n(\mathbb{C})$ is unitarily equivalent to $D_n(\mathbb{C})$, i.e., $U^*\mathcal{A}U=D_n(\mathbb{C})$ for some unitary operator $U$ in $M_n(\mathbb{C})$.
Thus, there exists an increasing sequence of projections $\{E_j\}_{j=1}^n$ such that
\begin{equation}\label{equ masa-In}
  \mathcal{A}=W^*(\{E_j\}_{j=1}^n)\quad\text{and}\quad
  \mathrm{Tr}(E_j)=j\quad\text{for each}~1\leqslant j\leqslant n.
\end{equation}
Similarly, let $(\mathcal{M},\tau)$ be a type $\mathrm{II}_1$ factor with a normal faithful tracial state $\tau$ and $\mathcal{A}$ a maximal abelian self-adjoint subalgebra of $\mathcal{M}$.
For the pair $(\mathcal{M},\mathcal{A})$, by applying \cite[Theorem 5.6.2]{SS-book}, there is an increasing family of projections $\{E_t\}_{0<t\leqslant 1}$ such that
\begin{equation}\label{equ masa-II1}
  \mathcal{A}=W^*(\{E_t\}_{0<t\leqslant 1})\quad\text{and}\quad
  \tau(E_t)=t\quad\text{for each}~0<t\leqslant 1.
\end{equation}
In the following lemma, we prove a similar result for maximal abelian self-adjoint subalgebras in properly infinite semifinite factors.
Note that if $\mathcal{A}$ is a maximal abelian self-adjoint subalgebra of $\mathcal{M}$, then $\mathcal{A}P$ is a maximal abelian self-adjoint subalgebra of $P\mathcal{M}P$ for every nonzero projection $P$ in $\mathcal{A}$.

\begin{lemma}\label{lem masa-semifinite}
    Let $(\mathcal{M},\tau)$ be a separable properly infinite semifinite factor with a normal faithful semifinite tracial weight $\tau$.
    Suppose that $\mathcal{A}$ is a maximal abelian self-adjoint subalgebra of $\mathcal{M}$ such that
    \begin{equation*}
        I=\bigvee\{Q\in\mathcal{A}\colon Q~\text{is a finite projection in}~\mathcal{M}\}.
    \end{equation*}
    \begin{enumerate} [label= $(\arabic*)$, ref= \arabic*, leftmargin=*] 
        \item  If $\mathcal{M}$ is of type $\mathrm{I}_\infty$ and 
        $\tau(P)=1$ for every minimal projection $P$ in $\mathcal{M}$, then there is an increasing sequence of projections $\{E_n\}_{n=1}^{\infty}$ with
        \begin{equation*}
            \mathcal{A}=W^*(\{E_n\}_{n=1}^{\infty}),\quad
            I=\bigvee\nolimits_{n\geqslant 1} E_n, \quad\text{and}\quad
            \tau(E_n)=n\quad\text{for each}~n\geqslant 1.
        \end{equation*}
        \item  If $\mathcal{M}$ is of type $\mathrm{II}_\infty$, then there is an increasing family of projections $\{E_t\}_{t>0}$ with
        \begin{equation*}
            \mathcal{A}=W^*(\{E_t\}_{t>0}),\quad
            I=\bigvee \nolimits_{t >0} E_t,\quad\text{and}\quad
            \tau(E_t)=t\quad\text{for each}~t>0.
        \end{equation*}
    \end{enumerate}
\end{lemma}

\begin{proof}
    By the separability of $\mathcal{A}$, there exists a sequence of finite projections $\{P_n\}_{n=1}^{\infty}$ in $\mathcal{A}$ such that $I=\bigvee_{n\geqslant 1} P_n$.
    Let
    \begin{equation*}
        Q_1=P_1\quad\text{and}\quad
        Q_n=P_n(I-Q_1-\cdots-Q_{n-1})\quad\text{for}~n\geqslant 2.
    \end{equation*}
    It is clear that $\{Q_n\}_{n=1}^{\infty}$ is a sequence of mutually orthogonal finite projections in $\mathcal{A}$ satisfying that $I=\sum_{n=1}^{\infty}Q_n$.
    By taking a subsequence, we may assume that each projection $Q_n$ is nonzero.

    If $\mathcal{M}$ is of type $\mathrm{I}_\infty$, then $k_n=\tau(Q_n)$ is a positive integer.
    Since $\mathcal{A}Q_n$ is a maximal abelian self-adjoint subalgebra of $Q_n\mathcal{M}Q_n\cong M_{k_n}(\mathbb{C})$, there is an increasing sequence of projections $\{Q_{n,j}\}_{j=1}^{k_n}$ in $\mathcal{A}Q_n$ such that $W_0^*(\{Q_{n,j}\}_{j=1}^{k_n})=\mathcal{A}Q_n$ and $\tau(Q_{n,j})=j$ for every $1\leqslant j\leqslant k_n$.
    Let
    \begin{equation*}
        E_m=
        \begin{cases}
            Q_{1,m}, & \mbox{if }~1\leqslant m\leqslant k_1, \\
            Q_1+\cdots+Q_n+Q_{n+1,j}, & \mbox{if }~m=k_1+\cdots+k_n+j, 1\leqslant j\leqslant k_{n+1}, n\geqslant 1.
        \end{cases}
    \end{equation*}
    Then $\{E_n\}_{n=1}^{\infty}$ has the desired properties.

    If $\mathcal{M}$ is of type $\mathrm{II}_\infty$, then $t_n=\tau(Q_n)$ is a positive scalar.
    Since $\mathcal{A}Q_n$ is a maximal abelian self-adjoint subalgebra of the type $\mathrm{II}_1$ factor $Q_n\mathcal{M}Q_n$, according to \eqref{equ masa-II1}, there is an increasing family of projections $\{Q_{n,t}\}_{0<t\leqslant t_n}$ in $\mathcal{A}Q_n$ such that $W_0^*(\{Q_{n,t}\}_{0<t\leqslant t_n})=\mathcal{A}Q_n$ and $\tau(Q_{n,t})=t$ for every $0<t\leqslant t_n$.
    Let
    \begin{equation*}
        E_s=
        \begin{cases}
            Q_{1,s}, & \mbox{if}~0<s\leqslant t_1, \\
            Q_1+\cdots+Q_n+Q_{n+1,t}, & \mbox{if}~s=t_1+\cdots+t_n+t, 0<t\leqslant t_{n+1}, n\geqslant 1.
        \end{cases}
    \end{equation*}
    Then $\{E_t\}_{t>0}$ has the desired properties.
    We complete the proof.
\end{proof}

\begin{remark} \label{rem masa-semifinite}
    By \cite[Theorem 2.3.7]{SS-book}, on a separable infinite-dimensional Hilbert space $\mathcal{H}$, there are two special maximal abelian self-adjoint subalgebras in $\mathcal{B}(\mathcal{H})$ which are not unitarily equivalent to each other. One is diagonal, while the other is diffuse.
    Actually, a maximal abelian self-adjoint subalgebra $\mathcal{A}$ of $\mathcal{B}(\mathcal{H})$ is diagonal if and only if $I$ is the union of all finite-dimensional projections in $\mathcal{A}$.
\end{remark}

In the following proposition, we show that, given a maximal abelian self-adjoint subalgebra $\mathcal{A}$ with a suitable condition, every partition of the identity operator $I$ is contained in $\mathcal{A}$ up to unitary equivalence.
Recall that all nonzero projections are equivalent in separable type $\mathrm{III}$ factors.

\begin{proposition}\label{prop partition}
    Let $\mathcal{M}$ be a separable factor and $\mathcal{A}$ a maximal abelian self-adjoint subalgebra of $\mathcal{M}$.
    When $\mathcal{M}$ is semifinite, we further assume that
    \begin{equation*}
        I=\bigvee\{Q\in\mathcal{A}\colon Q~\text{is a finite projection in}~\mathcal{M}\}.
    \end{equation*}
    Let $\{P_n\}_{n=1}^{\infty}$ be a sequence of projections in $\mathcal{M}$ such that $I=\sum_{n=1}^{\infty}P_n$.
    Then there exists a unitary operator $U$ in $\mathcal{M}$ such that $\{P_n\}_{n=1}^{\infty}\subseteq U\mathcal{A}U^*$.
\end{proposition}

\begin{proof}
    Suppose that $\mathcal{M}$ is of type $\mathrm{III}$.
    Without loss of generality, we assume that each $P_n$ is nonzero.
    Let $\{Q_n\}_{n=1}^{\infty}$ be a sequence of nonzero projections in $\mathcal{A}$ with $I=\sum_{n=1}^{\infty}Q_n$.
    Then there exists a partial isometry $U_n$ in $\mathcal{M}$ such that $U_nU_n^*=P_n$ and $U_n^*U_n=Q_n$ for every $n\geqslant 1$.
    We finish the proof by taking $U=\sum_{n=1}^{\infty}U_n$.

    In the remaining part of the proof, we only consider the case that $\mathcal{M}$ is of type $\mathrm{II}_\infty$, because other cases are similar.
    Since every projection is a sum of (possibly infinitely many) finite projections in $\mathcal{M}$, we may assume that each $P_n$ is a finite projection.
    Let $\{E_t\}_{t>0}$ be the increasing family of projections in $\mathcal{A}$ given by \Cref{lem masa-semifinite}.
    We define $Q_1=E_{\tau(P_1)}$ and
    \begin{equation*}
        Q_n=E_{\tau(P_1+\cdots+P_n)}-E_{\tau(P_1+\cdots+P_{n-1})}
        \quad\text{for}~n\geqslant 2.
    \end{equation*}
    Then $\{Q_n\}_{n=1}^{\infty}$ is a sequence of projections in $\mathcal{A}$ such that $I=\sum_{n=1}^{\infty}Q_n$ and $P_n\sim Q_n$ for each $n\geqslant 1$.
    The following part proceeds in the same way as the type $\mathrm{III}$ case.
\end{proof}

We end this section with the following technical lemma, which is prepared for the proof of \Cref{thm similarity}.

\begin{lemma} \label{lem P1-P2-P3}
    Let $\mathcal{M}$ be a separable factor, $\mathcal{A}$ an abelian von Neumann subalgebra of $\mathcal{M}$ such that $\dim\mathcal{A}>2$ and $P\precsim I-P$ in $\mathcal{M}$ for every minimal projection $P$ in $\mathcal{A}$.
    Then there are nonzero projections $\{P_j\}_{j=1}^{3}$ in $\mathcal{A}$ such that $I=\sum_{j=1}^{3}P_j$ and
    \begin{equation*}
        P_j\precsim I-P_j\quad\text{for every}~1\leqslant j\leqslant 3.
    \end{equation*}
\end{lemma}

\begin{proof}
    If there exists a minimal projection $P_1$ in $\mathcal{A}$ with $P_1\sim I-P_1$ in $\mathcal{M}$, then by the assumption $\dim\mathcal{A}>2$, there are nonzero projections $P_2$ and $P_3$ in $\mathcal{A}$ such that $I-P_1=P_2+P_3$.
    This completes the proof.
    Without loss of generality, we may assume that $P\prec I-P$ for every minimal projection $P$ in $\mathcal{A}$.
    The proof is divided into three cases.

    \noindent\textbf{Case I.} Suppose that $\mathcal{M}$ is of type $\mathrm{I}_n$, i.e., $\mathcal{M}\cong M_n(\mathbb{C})$.
    Let $\tau=\frac{1}{n}\mathrm{Tr}$, where $\mathrm{Tr}$ is the standard trace on $M_n(\mathbb{C})$.
    Then there are minimal projections $\{Q_j\}_{j=1}^k$ in $\mathcal{A}$ such that $I=\sum_{j=1}^{k}Q_j$.
    Moreover, $\tau(Q_j)<\frac{1}{2}$ for each $1\leqslant j\leqslant k$.
    It is clear that there exists an integer $k_1$ satisfying that $1\leqslant k_1<k$ and
    \begin{equation*}
        \tau(Q_1+\cdots+Q_{k_1})\leqslant\frac{1}{2}
        <\tau(Q_1+\cdots+Q_{k_1}+Q_{k_1+1}).
    \end{equation*}
    Since $\tau(Q_{k_1+1})<\frac{1}{2}$, we have $k_1+1<k$.
    Thus, we can take $P_1=Q_1+\cdots+Q_{k_1}$, $P_2=Q_{k_1+1}$, and $P_3=I-P_1-P_2$.

    \noindent \textbf{Case II.} Suppose that $\mathcal{M}$ is of type $\mathrm{II}_1$ with a normal faithful tracial state $\tau$.
    Let $P_0$ be the sum of all minimal projections in $\mathcal{A}$.
    Then $\mathcal{A}(I-P_0)$ is diffuse or $I-P_0=0$.
    If $\tau(P_0)>\frac{1}{2}$, then the proof is similar to \textbf{Case I}.
    If $0<\tau(P_0)\leqslant\frac{1}{2}$, then there are projections $P_2$ and $P_3$ in $\mathcal{A}(I-P_0)$ such that $I-P_0=P_2+P_3$ and $\tau(P_2)=\tau(P_3)$.
    We only need to take $P_1=P_0$.
    If $P_0=0$, then there are projections $\{P_j\}_{j=1}^{3}$ in $\mathcal{A}$ such that $I=\sum_{j=1}^{3}P_j$ and $\tau(P_j)=\frac{1}{3}$ for $1\leqslant j\leqslant 3$.

    \noindent\textbf{Case III.} Suppose that $\mathcal{M}$ is properly infinite.
    Let $\{Q_\lambda\}_{\lambda\in\Lambda}$ be a maximal orthogonal family of finite projections in $\mathcal{A}$ and $Q=\sum_{\lambda\in\Lambda}Q_\lambda$.
    Note that $Q=0$ when $\mathcal{M}$ is of type $\mathrm{III}$.
    If $Q=0$, then every nonzero projection in $\mathcal{A}$ is an infinite projection.
    Since $\dim\mathcal{A}>2$, there are nonzero projections $\{P_j\}_{j=1}^{3}$ in $\mathcal{A}$ such that $I=\sum_{j=1}^{3}P_j$.
    If $Q=I$, then there is a partition $\{\Lambda_1,\Lambda_2,\Lambda_3\}$ of $\Lambda$ such that $I\sim\sum_{\lambda\in\Lambda_j}Q_\lambda$ for $1\leqslant j\leqslant 3$.
    In this case, we take $P_j=\sum_{\lambda\in\Lambda_j}Q_\lambda$ for $1\leqslant j\leqslant 3$.
    If $Q$ is nontrivial, then $I-Q$ is an infinite projection.
    By the assumption that $P\prec I-P$ for every minimal projection $P$ in $\mathcal{A}$, $I-Q$ is not a minimal projection in $\mathcal{A}$.
    Let $P_2$ and $P_3$ be nonzero projections in $\mathcal{A}$ such that $I-Q=P_2+P_3$.
    Then both $P_2$ and $P_3$ are infinite projections.
    We can take $P_1=Q$.
    This completes the proof.
\end{proof}

\section{Special generators}\label{sec generator}

In this section, we present the proof of \Cref{thm masa}.
Before that, we prove a slightly stronger result in \Cref{prop U*PU-Pj}.
In the following lemma, we show that the matrix algebra $M_n(\mathbb{C})$ is generated by certain minimal projections for $n\geqslant 2$.
We will adopt the notation introduced in \Cref{def W0}.

\begin{lemma}\label{lem n=n1+n2}
    Let $(n_1,n_2)$ be a pair of positive integers and $n=n_1+n_2$.
    Then there are minimal projections $\{E_i\}_{i=1}^{n_1}$ and $\{F_j\}_{j=1}^{n_2}$ in $M_n(\mathbb{C})$ such that
    \begin{enumerate} [label= $(\arabic*)$, ref= \arabic*, leftmargin=*] 
        \item[$(1)$] $E_{i_1}E_{i_2}=0$ for $i_1\ne i_2$;
        \item[$(2)$] $F_{j_1}F_{j_2}=0$ for $j_1\ne j_2$;
        \item[$(3)$] $M_n(\mathbb{C})=W^*(\{E_i\}_{i=1}^{n_1},\{F_j\}_{j=1}^{n_2})$.
    \end{enumerate}
\end{lemma}

\begin{proof}
    We prove this lemma by induction. For $n_1=1$, $n_2=1$, and $n=2$, let
    \begin{equation*}
        E_1=
        \begin{pmatrix}
            1 & 0 \\
            0 & 0
        \end{pmatrix}\quad\text{and}\quad
        F_1=
        \begin{pmatrix}
            \frac{1}{2} & \frac{1}{2} \\
            \frac{1}{2} & \frac{1}{2}
        \end{pmatrix}.
    \end{equation*}
    Then $M_2(\mathbb{C}) = W^*(E_1, F_1)$.

    Suppose that there are minimal projections $\{E_i\}_{i=1}^{n_1}$ and $\{F_j\}_{j=1}^{n_2}$ in $M_n(\mathbb{C})$ with the desired properties.
    By symmetry, we only need to consider the pair $(n_1+1,n_2)$.
    Let $E_i'=E_i\oplus 0$ and $F_j'=F_j\oplus 0$ be minimal projections in $M_{n+1}(\mathbb{C})$.
    Then
    \begin{equation*}
        M_n(\mathbb{C})\oplus 0=W_0^*(\{E_i'\}_{i=1}^{n_1},\{F_j'\}_{j=1}^{n_2}).
    \end{equation*}
    Clearly, there exists a system of matrix units $\{E_{ij}\}_{i,j=1}^{2}$ in $M_{n+1}(\mathbb{C})$ such that $E_{11}$ is a minimal projection in $M_n(\mathbb{C})\oplus 0$ that is orthogonal to $\{E_i'\}_{i=1}^{n_1}$, and $E_{22}=0_n\oplus 1$.
    We construct a projection in $M_{n+1}(\mathbb{C})$ as follows
    \begin{equation*}
        E_{n_1+1}'=\frac{1}{2}(E_{11}+E_{12}+E_{21}+E_{22}).
    \end{equation*}
    Thus, $\{E_i'\}_{i=1}^{n_1+1}$ and $\{F_j'\}_{j=1}^{n_2}$ are minimal projections in $M_{n+1}(\mathbb{C})$ with the desired properties.
    This completes the proof.
\end{proof}

The following lemma is an analogue of \Cref{lem n=n1+n2} for separable type $\mathrm{II}_1$ factors $\mathcal{M}$ with $\mathcal{G}(\mathcal{M})=0$, where $\mathcal{G}$ denotes Shen's invariant \cite{Shen09}.

\begin{lemma}\label{lem 1=t1+t2}
    Let $(\mathcal{M},\tau)$ be a separable type $\mathrm{II}_1$ factor with $\mathcal{G}(\mathcal{M})=0$ and $(t_1,t_2)$ a pair of positive scalars with $1=t_1+t_2$.
    Then there are finitely many projections $\{P_i\}_{i=1}^{m_1}$ and $\{Q_j\}_{j=1}^{m_2}$ in $\mathcal{M}$ such that
    \begin{enumerate} [label= $(\arabic*)$, ref= \arabic*, leftmargin=*] 
        \item[$(1)$] $P_{i_1}P_{i_2}=0$ for $i_1\ne i_2$, and $t_1=\sum_{i=1}^{m_1}\tau(P_i)$;
        \item[$(2)$] $Q_{j_1}Q_{j_2}=0$ for $j_1\ne j_2$, and $t_2=\sum_{j=1}^{m_2}\tau(Q_j)$;
        \item[$(3)$] $\mathcal{M} = W^*(\{P_i\}_{i=1}^{m_1},\{Q_j\}_{j=1}^{m_2})$.
    \end{enumerate}
\end{lemma}

\begin{proof}
    Let $m$ be a positive integer such that $t_1\geqslant\frac{1}{m}$ and $t_2\geqslant\frac{1}{m}$.
    Then we can write
    \begin{equation*}
        t_1=\frac{n_1+s_1}{m}\quad\text{and}\quad t_2=\frac{n_2+s_2}{m},
    \end{equation*}
    where $n_1,n_2$ are positive integers, and $s_1,s_2\in [0,1)$.
    Without loss of generality, we assume that $\mathcal{M} = M_m(\mathbb{C})\otimes\mathcal{N}$, where $\mathcal{N}$ is a separable type $\mathrm{II}_1$ factor. By applying \eqref{equ SS-16.4.5}, we obtain that $\mathcal{G}(\mathcal{N}) = 0$. Moreover, it follows from  \Cref{cor M4} that there are projections $R_1$, $R_2$, and $R_3$ in $\mathcal{N}$ such that
    \begin{equation*}
        \mathcal{N} = W^*(R_1,R_2,R_3)\quad\text{and}\quad R_2R_3=0.
    \end{equation*}
    Write $n=n_1+n_2\leqslant m$.
    By \Cref{lem n=n1+n2}, there are minimal projections $\{E_i\}_{i=1}^{n_1}$ and $\{F_j\}_{j=1}^{n_2}$ in $M_m(\mathbb{C})$ such that
    \begin{enumerate} [label= $(\arabic*)$, ref= \arabic*, leftmargin=*] 
        \item[$(1')$] $E_{i_1}E_{i_2}=0$ for $i_1\ne i_2$;
        \item[$(2')$] $F_{j_1}F_{j_2}=0$ for $j_1\ne j_2$;
        \item[$(3')$] $M_n(\mathbb{C})\oplus 0_{m-n} = W_0^*(\{E_i\}_{i=1}^{n_1},\{F_j\}_{j=1}^{n_2})$.
    \end{enumerate}
    We construct two families of projections in $\mathcal{M}$ as follows:
    \begin{align*}
        \{P_i\}_{i=1}^{2n_1} &= \{E_i\otimes R_1,E_i\otimes (I_{\mathcal{N}}-R_1)\colon 1\leqslant i\leqslant n_1\},\\
        \{Q_j\}_{j=1}^{3n_2} &= \{F_j\otimes R_2,F_j\otimes R_3,F_j\otimes(I_{\mathcal{N}}-R_2-R_3)\colon 1\leqslant j\leqslant n_2\}.
    \end{align*}
    Write $P=(I_n\oplus 0_{m-n})\otimes I_{\mathcal{N}}$.
    It is clear that
    \begin{equation*}
        P\mathcal{M}P = W_0^*(\{P_i\}_{i=1}^{2n_1},\{Q_j\}_{j=1}^{3n_2}).
    \end{equation*}
    Since $\tau(P-\sum_{i=1}^{2n_1}P_i) = \frac{n}{m}-\frac{n_1}{m}=\frac{n_2}{m}\geqslant\frac{s_1}{m}$ and $\tau(I-P)=\frac{s_1+s_2}{m}\geqslant\frac{s_1}{m}$, there exists a system of matrix units $\{E_{ij}\}_{i,j=1}^2$ in $\mathcal{M}$ such that
    \begin{equation*}
        \tau(E_{11})=\tau(E_{22})=\frac{s_1}{m},\quad
        E_{11}\leqslant P-\sum_{i=1}^{2n_1}P_i,\quad\text{and}\quad
        E_{22}\leqslant I-P.
    \end{equation*}
    Define $P_{2n_1+1}=\frac{1}{2}(E_{11}+E_{12}+E_{21}+E_{22})$ and $Q=P+E_{22}$.
    Then
    \begin{equation*}
        Q\mathcal{M}Q = W_0^*(\{P_i\}_{i=1}^{2n_1+1},\{Q_j\}_{j=1}^{3n_2}).
    \end{equation*}
    Similarly, since $\tau(Q-\sum_{j=1}^{3n_2}Q_j)
    =\frac{n+s_1}{m}-\frac{n_2}{m}=\frac{n_1+s_1}{m}\geqslant\frac{s_2}{m}$ and $\tau(I-Q)=\frac{s_2}{m}$, there exists a system of matrix units $\{F_{ij}\}_{i,j=1}^2$ in $\mathcal{M}$ such that
    \begin{equation*}
        \tau(F_{11}) = \tau(F_{22})=\frac{s_2}{m},\quad
        F_{11}\leqslant Q-\sum_{j=1}^{3n_2}Q_j,\quad\text{and}\quad
        F_{22}=I-Q.
    \end{equation*}
    Define $Q_{3n_2+1} = \frac{1}{2}(F_{11}+F_{12}+F_{21}+F_{22})$.
    Then
    \begin{equation*}
        \mathcal{M} = W_0^*(\{P_i\}_{i=1}^{2n_1+1},\{Q_j\}_{j=1}^{3n_2+1}).
    \end{equation*}
    We complete the proof by taking $m_1=2n_1+1$ and $m_2=3n_2+1$.
\end{proof}

As an application of \Cref{lem 1=t1+t2}, we show that any decomposition of $1$ into positive scalars can be realized as the traces of the range projections of positive operators that generates $\mathcal{M}$ under the assumption $\mathcal{G}(\mathcal{M})=0$.

\begin{lemma}\label{lem sum=1}
    Let $\mathcal{M}$ be a separable type $\mathrm{II}_1$ factor with $\mathcal{G}(\mathcal{M})=0$ and $\{t_n\}_{n=1}^{N}$ a sequence of positive scalars with $1=\sum_{n=1}^{N}t_n$, where $2\leqslant N\leqslant\infty$.
    Then there exists a sequence of positive operators $\{A_n\}_{n=1}^{N}$ in $\mathcal{M}$ such that $\mathcal{M}=W^*(\{A_n\}_{n=1}^N)$ and $\tau(R(A_n))=t_n$ for every $n\geqslant 1$.
\end{lemma}

\begin{proof}
    By \Cref{lem 1=t1+t2}, there are finitely many projections $\{P_i\}_{i=1}^{m_1}$ and $\{Q_j\}_{j=1}^{m_2}$ in $\mathcal{M}$ such that
    \begin{enumerate} [label= $(\arabic*)$, ref= \arabic*, leftmargin=*] 
        \item[$(1)$] $P_{i_1}P_{i_2}=0$ for $i_1\ne i_2$, and $t_1=\sum_{i=1}^{m_1}\tau(P_i)$;
        
        \item[$(2)$] $Q_{j_1}Q_{j_2}=0$ for $j_1\ne j_2$, and $1-t_1=\sum_{j=1}^{m_2}\tau(Q_j)$;
        
        \item[$(3)$] $\mathcal{M} = W^*(\{P_i\}_{i=1}^{m_1},\{Q_j\}_{j=1}^{m_2})$.
    \end{enumerate}
    Define
    \begin{equation*}
        A_1 = \sum_{j=1}^{m_1} jP_j,\quad 
        B =\sum_{j=1}^{m_2} jQ_j,\quad 
        P =\sum_{j=1}^{m_1}P_j, \quad\text{and}\quad 
        Q =\sum_{j=1}^{m_2}Q_j.
    \end{equation*}
    Then $\mathcal{M}=W^*(A_1,B)$, $R(A_1)=P$, and $R(B)=Q$.
    Since $B$ is a positive operator in $Q\mathcal{M}Q$, it is contained in some maximal abelian self-adjoint subalgebra $\mathcal{A}$ of $Q\mathcal{M}Q$.
    Clearly, $\mathcal{A}$ is a diffuse von Neumann algebra and there are projections $\{R_n\}_{n=2}^{N}$ in $\mathcal{A}$ satisfying that $Q=\sum_{n=2}^{N}R_n$ and $\tau(R_n)=t_n$ for every $n\geqslant 2$.
    The proof is complete by taking $A_n=BR_n$ for every $n\geqslant 2$.
\end{proof}

\begin{remark}\label{rem sum=1}
    Combining \Cref{lem n=n1+n2} and the proof of \Cref{lem sum=1}, we can obtain the following result:
    Let $\{n_j\}_{n=1}^{N}$ be finitely many positive integers and $n=\sum_{j=1}^{N}n_j$.
    Then there are positive operators $\{A_j\}_{j=1}^{N}$ in $M_n(\mathbb{C})$ such that $M_n(\mathbb{C})=W^*(\{A_j\}_{j=1}^N)$ and $\operatorname{rank}(A_j)=n_j$ for every $1\leqslant j\leqslant N$.
\end{remark}

If we relax the condition $1=\sum_{n=1}^{N}t_n$ in \Cref{lem sum=1}, then we can drop the assumption $\mathcal{G}(\mathcal{M})=0$ and obtain the following result.

\begin{lemma} \label{lem sum=infty}
    Let $(\mathcal{M},\tau)$ be a separable type $\mathrm{II}_1$ factor and $\{t_n\}_{n=1}^{\infty}$ a sequence of positive scalars such that $\sum_{n=1}^{\infty}t_n=\infty$ and $0<t_n\leqslant\frac{1}{2}$ for every $n\geqslant 1$.
    Then there exists a sequence of projections $\{P_n\}_{n=1}^{\infty}$ in $\mathcal{M}$ such that $\mathcal{M}=W^*(\{P_n\}_{n=1}^{\infty})$ and $\tau(P_n)=t_n$ for every $n\geqslant 1$.
\end{lemma}

\begin{proof}
    Let $n_0=0$ for simplicity.
    Since $0<t_n\leqslant\frac{1}{2}$, there exists $n_1>n_0$ such that $s_1=t_{n_0+1}+t_{n_0+2}+\cdots+t_{n_1}\in[\frac{1}{4},\frac{3}{4}]$.
    Inductively, for every $j\geqslant 2$, there exists $n_j>n_{j-1}$ such that
    \begin{equation*}
        s_j=t_{n_{j-1}+1}+t_{n_{j-1}+2}+\cdots+t_{n_j}
        \in\left[\frac{1}{4},\frac{3}{4}\right].
    \end{equation*}
    Clearly, there exists a convergent subsequence $\{s_{j_k}\}_{k=1}^{\infty}$ of $\{s_j\}_{j=1}^{\infty}$ such that $s_{j_k}\to c$ as $k\to\infty$ for some constant $c\in[\frac{1}{4},\frac{3}{4}]$.
    
    Recall that the strong-operator topology and the $\|\cdot\|_2$-norm topology coincide on the closed unit ball of $\mathcal{M}$, where $\|T\|_2=\tau(T^*T)^{\frac{1}{2}}$ for $T\in\mathcal{M}$.
    Let $\mathcal{P}_c$ be the set of all projections in $\mathcal{M}$ with trace $c$, $\{E_k\}_{k=1}^{\infty}$ a strong-operator dense sequence in $\mathcal{P}_c$, and $\{Q_j\}_{j=1}^{\infty}$ a sequence of projections in $\mathcal{M}$ such that
    \begin{enumerate}[label=$(\arabic*)$, ref=\arabic*, leftmargin=*]
        \item[$(1)$] $\tau(Q_j)=s_j$ for every $j\geqslant 1$;
        \item[$(2)$] either $Q_{j_k}\leqslant E_k$ or $E_k\leqslant Q_{j_k}$ for every $k\geqslant 1$.
    \end{enumerate}
    Then $\|Q_{j_k}-E_k\|_2=|s_{j_k}-c|^{\frac{1}{2}}\to 0$ as $k\to\infty$.
    Since $\{E_k\}_{k=1}^{\infty}$ is strong-operator dense in $\mathcal{P}_c$, the set $\mathcal{P}_c$ is contained in the strong-operator closure of $\{Q_{j_k}\}_{k=1}^{\infty}$.
    It follows that
    \begin{equation*}
        \mathcal{M}=W^*(\mathcal{P}_c)\subseteq W^*(\{Q_{j_k}\}_{k=1}^{\infty})\subseteq W^*(\{Q_j\}_{j=1}^{\infty}).
    \end{equation*}
    Therefore, we have $\mathcal{M}=W^*(\{Q_j\}_{j=1}^{\infty})$.
    Let $\{P_n\}_{n=1}^{\infty}$ be a sequence of projections in $\mathcal{M}$ such that
    \begin{enumerate}[label=$(\arabic*)$, ref=\arabic*, leftmargin=*]
        \item[$(1)$] $\tau(P_n)=t_n$ for every $n\geqslant 1$;
        \item[$(2)$] $Q_j=P_{n_{j-1}+1}+P_{n_{j-1}+2}+\cdots+P_{n_j}$ for every $j\geqslant 1$.
    \end{enumerate}
    Then $\{P_n\}_{n=1}^{\infty}$ has the desired properties.    
\end{proof}

The following corollary is a direct consequence of \Cref{lem sum=infty}.

\begin{corollary}\label{cor sum=infty}
    Let $(\mathcal{M},\tau)$ be a separable type $\mathrm{II}_1$ factor and $\{t_n\}_{n=1}^{\infty}$ a sequence of positive scalars such that $\sum_{n=1}^{\infty}t_n=\infty$ and $0<t_n\leqslant 1$ for every $n\geqslant 1$.
    Then there exists a sequence of positive operators $\{A_n\}_{n=1}^{\infty}$ in $\mathcal{M}$ such that $\mathcal{M}=W^*(\{A_n\}_{n=1}^{\infty})$ and $\tau(R(A_n))=t_n$ for every $n\geqslant 1$.
\end{corollary}

\begin{proof}
    By \Cref{lem sum=infty}, there exists a sequence of projections $\{P_n\}_{n=1}^{\infty}$ in $\mathcal{M}$ such that $\mathcal{M}=W^*(\{P_n\}_{n=1}^{\infty})$ and $\tau(P_n)=\frac{1}{2}t_n$ for every $n\geqslant 1$.
    Let $Q_n$ be a projection in $\mathcal{M}$ such that
    \begin{equation*}
        P_nQ_n=0\quad\text{and}\quad \tau(Q_n)=\frac{1}{2}t_n.    
    \end{equation*}
    We complete the proof by taking $A_n=P_n+2Q_n$.
\end{proof}

Given a partition $\{P_n\}_{n=0}^{N}$ of the identity $I$, the von Neumann algebra generated by $\{P_n\}_{n=0}^{N}$ is abelian.
We show that it is possible to find a projection $P_0'$ unitarily equivalent to $P_0$ such that the von Neumann algebra generated by $\{P_0'\}\cup\{P_n\}_{n=1}^{N}$ is $\mathcal{M}$.
The following lemma is prepared for \Cref{prop U*PU-Pj}.

\begin{lemma}\label{lem key-construction}
    Let $\mathcal{M}$ be a separable factor and $\{P_n\}_{n=0}^{N}$ a sequence of nonzero projections in $\mathcal{M}$ such that
    \begin{enumerate}[label=$(\arabic*)$, ref=\arabic*, leftmargin=*]
        \item[$(1)$] $I=\sum_{n=0}^{N}P_n$ and $2\leqslant N\leqslant\infty$;
        \item[$(2)$] $P_n\precsim P_0\precsim I-P_0$ for every $n\geqslant 1$.
    \end{enumerate}
    Suppose that there exists a sequence of positive operators $\{A_n\}_{n=1}^{N}$ in $P_0\mathcal{M}P_0$ such that $P_0\mathcal{M}P_0=W_0^*(\{A_n\}_{n=1}^{N})$ and $R(A_n)\sim P_n$ in $\mathcal{M}$ for every $n\geqslant 1$.
    Then there exists a unitary operator $U$ in $\mathcal{M}$ such that
    \begin{equation*}
        \mathcal{M}=W^*(U^*P_0U,\{P_n\}_{n=1}^{N}).
    \end{equation*}
\end{lemma}

\begin{proof}
    Without loss of generality, we may assume that $\|A_n\|\leqslant 2^{-n}$ for every $n\geqslant 1$.
    Let $A_0=(P_0-\sum_{n=1}^{N}A_n^2)^{\frac{1}{2}}$.
    Then $A_0$ is an invertible operator in $P_0\mathcal{M}P_0$, whose inverse is denoted by $B_0$, i.e., $A_0B_0=B_0A_0=P_0$.
    Since $R(A_n)\sim P_n$, there exists a system of matrix units $\{E_{ij}^{(n)}\}_{i,j=1}^{2}$ in $\mathcal{M}$ such that
    \begin{equation*}
        E_{11}^{(n)}=R(A_n)\quad\text{and}\quad E_{22}^{(n)}=P_n.
    \end{equation*}
    Let $V=A_0+\sum_{n=1}^{N}A_nE_{12}^{(n)}$.
    It is straightforward to verify that
    \begin{equation*}
        VV^*=A_0^2+\sum_{n=1}^{N}A_n^2=P_0.
    \end{equation*}
    Hence $V$ is a partial isometry in $\mathcal{M}$.
    Let $P_0'=V^*V$ and $\mathcal{N}=W^*(P_0',\{P_n\}_{n=1}^{N})$.
    It remains to show that $P_0'$ is unitarily equivalent to $P_0$ and $\mathcal{M}=\mathcal{N}$.
    
    If $P_0(\sim P_0')$ is a finite projection, then $I-P_0\sim I-P_0'$.
    If $P_0$ is an infinite projection, then $I-P_0=\sum_{n=1}^{N}P_n$ is an infinite projection.
    Note that
    \begin{equation*}
        P_n(I-P_0')P_n=E_{21}^{(n)}(P_0-A_n^2)E_{12}^{(n)}
        \geqslant\frac{1}{2}E_{21}^{(n)}P_0E_{12}^{(n)}
        =\frac{1}{2}P_n\quad\text{for every}~n\geqslant 1.
    \end{equation*}
    It follows that $I-P_0'$ is an infinite projection, and hence $I-P_0\sim I-P_0'$.
    Thus, $P_0$ and $P_0'$ are unitarily equivalent in $\mathcal{M}$.
    
    Recall that $\mathcal{N}=W^*(P_0',\{P_n\}_{n=1}^{N})$.
    Then $P_0=I-\sum_{n=1}^{N}P_n\in\mathcal{N}$.
    It follows that $P_0P_0'P_0=A_0^2\in\mathcal{N}$, and hence $A_0\in\mathcal{N}$.
    Since $B_0$ is the inverse of $A_0$ in $P_0\mathcal{M}P_0$, we have $B_0\in\mathcal{N}$.
    For every $n\geqslant 1$, we have
    \begin{equation*}
        B_0P_0P_0'P_n=B_0A_0A_nE_{12}^{(n)}= A_nE_{12}^{(n)}\in\mathcal{N}.
    \end{equation*}
    By polar decomposition, we see that $A_n,E_{12}^{(n)}\in\mathcal{N}$.
    It follows that
    \begin{equation*}
        P_0\mathcal{M}P_0= W_0^*(\{A_n\}_{n=1}^{N})\subseteq\mathcal{N}.
    \end{equation*}
    In particular, $P_0\mathcal{M}P_0=P_0\mathcal{N}P_0$.
    For every $n\geqslant 1$, we have
    \begin{equation*}
        P_0\mathcal{M}P_n=P_0\mathcal{M}P_0E_{12}^{(n)}P_n
        =P_0\mathcal{N}P_0E_{12}^{(n)}P_n=P_0\mathcal{N}P_n.
    \end{equation*}
    Thus, $P_n\mathcal{M}P_0=(P_0\mathcal{M}P_n)^*=(P_0\mathcal{N}P_n)^*=P_n\mathcal{N}P_0$.
    Similarly, for every $m,n\geqslant 1$, we have
    \begin{equation*}
        P_m\mathcal{M}P_n= P_mE_{21}^{(m)}P_0\mathcal{M}P_0E_{12}^{(n)}P_n
        = P_mE_{21}^{(m)}P_0\mathcal{N}P_0E_{12}^{(n)}P_n
        = P_m\mathcal{N}P_n.
    \end{equation*}
    We complete the proof.        
\end{proof}

\begin{proposition}\label{prop U*PU-Pj}
    Let $\mathcal{M}$ be a separable factor and $\{P_j\}_{j=0}^{N}$ a sequence of nonzero projections in $\mathcal{M}$ such that
    \begin{enumerate}[label=$(\arabic*)$, ref=\arabic*, leftmargin=*] 
        \item[$(1)$] $I=\sum_{j=0}^{N}P_j$ and $2\leqslant N\leqslant\infty$;
        \item[$(2)$] $P_j\precsim P_0\precsim I-P_0$ for every $j\geqslant 1$.
    \end{enumerate}
    We further assume that $\mathcal{G}(\mathcal{M})=0$ if $\mathcal{M}$ is of type $\mathrm{II}_1$.
    Then there exists a unitary operator $U$ in $\mathcal{M}$ such that
    \begin{equation*}
        \mathcal{M}=W^*(U^*P_0U,\{P_j\}_{j=1}^{N}).
    \end{equation*}
\end{proposition}

\begin{proof}
    By \Cref{lem key-construction}, it suffices to find positive operators $\{A_j\}_{j=1}^{N}$ in $P_0\mathcal{M}P_0$ such that $P_0\mathcal{M}P_0=W_0^*(\{A_j\}_{j=1}^{N})$ and $R(A_j)\sim P_j$ in $\mathcal{M}$ for every $j\geqslant 1$.

    \noindent\textbf{Case I.}
    Suppose that $\mathcal{M}$ is of type $\mathrm{I}_n$, i.e., $\mathcal{M}\cong M_n(\mathbb{C})$.
    Let $n_j=\mathrm{Tr}(P_j)$ for every $j\geqslant 0$.
    Then $n_0\leqslant\sum_{j=1}^{N}n_j$ and $n_j\leqslant n_0$ for every $j\geqslant 1$.
    By \Cref{rem sum=1}, there are positive operators $\{B_j\}_{j=1}^{N}$ in $P_0\mathcal{M}P_0$ such that
    \begin{equation*}
        P_0\mathcal{M}P_0=W_0^*(\{B_j\}_{j=1}^N)
    \end{equation*}
    and $\operatorname{rank}(B_j)\leqslant n_j$ for every $j\geqslant 1$.
    Let $Q_j$ be a projection in $M_n(\mathbb{C})$ such that
    \begin{equation*}
        Q_jB_j=0\quad\text{and}\quad \mathrm{Tr}(Q_j)=n_j-\operatorname{rank}(B_j).
    \end{equation*}
    Let $A_j=B_j+(\|B_j\|+1)Q_j$ for every $1\leqslant j\leqslant N$.
    Then $\operatorname{rank}(A_j)=n_j=\mathrm{Tr}(P_j)$.
    Hence $\{A_j\}_{j=1}^{N}$ has the desired properties.
    
    \noindent\textbf{Case II.}
    Suppose that $\mathcal{M}$ is of type $\mathrm{II}_1$ with $\mathcal{G}(\mathcal{M})=0$.
    Let $t_j=\tau(P_j)$ for every $j\geqslant 0$.
    Then $t_0\leqslant\sum_{j=1}^{N}t_j$ and $t_j\leqslant t_0$ for every $j\geqslant 1$.
    By \Cref{lem sum=1}, there are positive operators $\{B_j\}_{j=1}^{N}$ in $P_0\mathcal{M}P_0$ such that
    \begin{equation*}
        P_0\mathcal{M}P_0=W_0^*(\{B_j\}_{j=1}^N)
    \end{equation*}
    and $\tau(R(B_j))\leqslant t_j$ for every $j\geqslant 1$.
    The remaining part is similar to \textbf{Case I}.

    In the following cases, we assume that $\mathcal{M}$ is a properly infinite factor.
    
    \noindent\textbf{Case III.}
    Suppose that $P_0$ is a finite projection in $\mathcal{M}$.
    It is clear that $N=\infty$ and $\mathcal{M}$ is of type $\mathrm{I}_\infty$ or type $\mathrm{II}_\infty$.
    If $\mathcal{M}$ is of type $\mathrm{II}_\infty$, then $P_0\mathcal{M}P_0$ is of type $\mathrm{II}_1$.
    Thus, we can find such $\{A_j\}_{j=1}^{\infty}$ by \Cref{cor sum=infty}.
    If $\mathcal{M}$ is of type $\mathrm{I}_\infty$, then $P_0\mathcal{M}P_0$ is *-isomorphic to $M_{n_0}(\mathbb{C})$, where $n_0=\mathrm{Tr}(P_0)$.
    Therefore, we can easily find such $\{A_j\}_{j=1}^{\infty}$ as $M_{n_0}(\mathbb{C})$ is generated by finitely many minimal projections.
    
    \noindent\textbf{Case IV.}
    Suppose that $P_0$ is an infinite projection in $\mathcal{M}$.
    Then $P_0\mathcal{M}P_0$ is a properly infinite factor.
    By \cite{Wog69}, there are invertible positive operators $A$ and $B$ in $P_0\mathcal{M}P_0$ such that $P_0\mathcal{M}P_0=W_0^*(A,B)$.
    If $P_j$ is an infinite projection for every $j\geqslant 1$, then we take $A_1=A$ and $A_j=B$ for every $j\geqslant 2$.
    It is clear that $\{A_j\}_{j=1}^{N}$ has the desired properties.
    Without loss of generality, we may assume that $P_1$ is a finite projection.
    For each $j\geqslant 2$, there exists a family of nonzero projections $\{P_{jk}\}_{k\in\Lambda_j}$ in $\mathcal{M}$ such that $P_j=\sum_{k\in\Lambda_j}P_{jk}$ and $P_{jk}\precsim P_1$ for every $k\in\Lambda_j$, where $\Lambda_j\subseteq\mathbb{N}$.
    Then
    \begin{equation*}
        I-P_0=\sum_{j=1}^{N}P_j=P_1+\sum_{j=2}^{N}\sum_{k\in\Lambda_j}P_{jk}.
    \end{equation*}
    Since both $P_0$ and $I-P_0$ are infinite projections, there exists a partial isometry $V$ in $\mathcal{M}$ such that $P_0=V^*V$ and $I-P_0=VV^*$.
    Let $Q_1=V^*P_1V$ and $Q_{jk}=V^*P_{jk}V$.
    By \textbf{Case III}, there exists a unitary operator $U_0$ in $P_0\mathcal{M}P_0$ such that
    \begin{equation*}
        P_0\mathcal{M}P_0=W_0^*(U_0^*Q_1U_0,\{Q_{jk}\colon k\in\Lambda_j,2\leqslant j\leqslant N\}).
    \end{equation*}
    Let $A_1=U_0^*Q_1U_0$ and $A_j=\sum_{k\in\Lambda_j}\frac{1}{2^k}Q_{jk}$ for every $j\geqslant 2$.
    It is routine to verify that $\{A_j\}_{j=1}^{N}$ has the desired properties.
\end{proof}

\begin{remark}\label{rem U*PU-Pj}
    We illustrate the necessity of the conditions $(1)$ and $(2)$ in \Cref{prop U*PU-Pj}.
    Let $Q=U^*P_0U$.
    Then $\mathcal{M}=W^*(Q,\{P_j\}_{j=1}^{N})$.
    Since the von Neumann algebra generated by two projections can only have direct summands of type $\mathrm{I}_1$ and $\mathrm{I}_2$, the condition $N\geqslant 2$ is necessary.
    Since $Q-R(Q(I-P_0))$ is a central projection, it follows that $Q=R(Q(I-P_0))$ and hence
    \begin{equation*}
        P_0\sim Q=R(Q(I-P_0))\sim R((I-P_0)Q)\leqslant I-P_0.
    \end{equation*}
    Similarly, since $P_j-R(P_jQ)$ is a central projection for every $j\geqslant 1$, we have
    \begin{equation*}
        P_j=R(P_jQ)\sim R(QP_j)\leqslant Q\sim P_0.
    \end{equation*}
    Therefore, the condition $P_j\precsim P_0\precsim I-P_0$ is also necessary.
\end{remark}

Now we are ready to present the proof of \Cref{thm masa}.

\begin{proof}[Proof of \Cref{thm masa}]
    Since the conclusion is clearly true for $\mathbb{C}$ and $M_2(\mathbb{C})$, we only need to consider the case that $\dim\mathcal{M}>4$.
    Let $P$ be a nonzero projection in $\mathcal{A}$ with $P\precsim I-P$ in $\mathcal{M}$.
    We claim that there exists a sequence of nonzero projections $\{P_j\}_{j=0}^{N}$ in $\mathcal{A}$ such that
    \begin{enumerate}[label=$(\arabic*)$, ref=\arabic*, leftmargin=*]
        \item[$(1')$] $I-P=\sum_{j=1}^{N}P_j$ and $2\leqslant N\leqslant\infty$;
        \item[$(2')$] $P_j\precsim P$ for every $j\geqslant 1$.
    \end{enumerate}
    If $P$ is an infinite projection, then we can easily find such $\{P_j\}_{j=0}^{N}$ because the condition $(2')$ is automatically true.
    If $P$ is a finite projection, then \eqref{equ masa-In}, \eqref{equ masa-II1}, and
    \Cref{lem masa-semifinite} guarantee the existence of such $\{P_j\}_{j=1}^{N}$.
    By \Cref{prop U*PU-Pj}, there exists a unitary operator $U$ in $\mathcal{M}$ such that
    \begin{equation*}
        \mathcal{M}=W^*(U^*PU,\{P_j\}_{j=1}^{N})\subseteq W^*(U^*PU,\mathcal{A}(I-P)).
    \end{equation*}
    This completes the proof.
\end{proof}

At the end of this section, we list several consequences of \Cref{thm masa} and \Cref{prop U*PU-Pj}.
The following proposition is a generalization of \cite[Theorem 2]{FT69}.

\begin{proposition}\label{prop FT}
    Let $\mathcal{M}$ be a separable factor with $\dim\mathcal{M}>1$.
    We further assume that $\mathcal{G}(\mathcal{M})=0$ if $\mathcal{M}$ is of type $\mathrm{II}_1$.
    Then a self-adjoint operator in $\mathcal{M}$ is the real part of a generator if and only if it is not a scalar multiple of the identity.
\end{proposition}

\begin{proof}
    The sufficiency is clear and we will prove the necessity.
    Let $A$ be a self-adjoint operator in $\mathcal{M}$ which is not a scalar.
    Then it has a nonzero spectral projection $P$ with $P\precsim I-P$.
    Let $\mathcal{A}$ be a maximal abelian self-adjoint subalgebra of $\mathcal{M}$ containing $P$ such that
    \begin{equation*}
        I=\bigvee\{Q\in\mathcal{A}\colon Q~\text{is a finite projection in}~\mathcal{M}\}
    \end{equation*}
    if $\mathcal{M}$ is of type $\mathrm{I}_\infty$ or type $\mathrm{II}_\infty$.
    By \Cref{thm masa}, there exists a unitary operator $U$ in $\mathcal{M}$ such that
    \begin{equation*}
        \mathcal{M}=W^*(U^*PU,\mathcal{A}(I-P))=W^*(P,U\mathcal{A}(I-P)U^*).
    \end{equation*}
    Let $B$ be a positive operator in $\mathcal{M}$ such that $W^*(B)=W^*(U\mathcal{A}(I-P)U^*)$.
    It is clear that $\mathcal{M}=W^*(A,B)$.
    This completes the proof.
\end{proof}

\begin{remark}\label{rem FT}
    By using \Cref{prop U*PU-Pj}, we may require that the operator $B$ has at most countable spectrum in the proof of \Cref{prop FT}.
\end{remark}

The following result is a direct consequence of \Cref{prop FT}.

\begin{corollary}\label{cor FT}
    Let $\mathcal{M}$ be a separable factor with $\dim\mathcal{M}>1$.
    Then a self-adjoint operator in $\mathcal{M}$ is the real part of an irreducible operator if and only if it is not a scalar multiple of the identity.
\end{corollary}

\begin{proof}
    Note that every generator of $\mathcal{M}$ is irreducible.
    By \Cref{prop FT}, we may assume that $\mathcal{M}$ is of type $\mathrm{II}_1$.
    Then $\mathcal{M}$ contains an irreducible hyperfinite subfactor $\mathcal{R}$ by \cite[Corollary 4.1]{Pop81}.
    Let $A$ be a self-adjoint operator in $\mathcal{M}$ which is not a scalar and $P$ a nontrivial spectral projection of $A$.
    It is clear that $P\in U^*\mathcal{R}U$ for some unitary operator $U$ in $\mathcal{M}$.
    Without loss of generality, we may assume that $P\in\mathcal{R}$.
    Since $\mathcal{G}(\mathcal{R})=0$ by \cite[Theorem 3.4]{Shen09}, there exists an self-adjoint operator $B$ in $\mathcal{R}$ such that $\mathcal{R}=W^*(P,B)$ by \Cref{prop FT}, and hence $\mathcal{R}\subseteq W^*(A,B)$.
    Therefore, $W^*(A,B)$ is an irreducible subfactor of $\mathcal{M}$.
    This completes the proof.
\end{proof}

The following corollary is a refined version of \Cref{lem sum=1}, which asserts that any decomposition of $1$ into positive scalars no more than $\frac{1}{2}$ can be realized as the traces of projections that generates $\mathcal{M}$ under the assumption $\mathcal{G}(\mathcal{M})=0$.

\begin{corollary}\label{cor sum=1}
    Let $\mathcal{M}$ be a separable type $\mathrm{II}_1$ factor with $\mathcal{G}(\mathcal{M})=0$.
    Suppose that $\{t_n\}_{n=0}^{N}$ is a sequence of positive scalars such that $1=\sum_{n=0}^{N}t_n$ and
    \begin{equation*}
        t_n\leqslant t_0\leqslant\frac{1}{2}\quad\text{for every}~n\geqslant 1,
    \end{equation*}
    where $2\leqslant N\leqslant\infty$.
    Then there exists a sequence of projections $\{P_n\}_{n=0}^{N}$ in $\mathcal{M}$ such that
    \begin{enumerate} [label= $(\arabic*)$, ref= \arabic*, leftmargin=*] 
        \item[$(1)$] $\tau(P_n)=t_n$ for every $n\geqslant 0$;
        \item[$(2)$] the projections $\{P_n\}_{n=1}^{N}$ are mutually orthogonal;
        \item[$(3)$] $\mathcal{M}=W^*(\{P_n\}_{n=0}^{N})$.
    \end{enumerate}
\end{corollary}

\begin{proof}
    Let $\{Q_n\}_{n=0}^{N}$ be a sequence of projections in $\mathcal{M}$ such that $I=\sum_{n=0}^{N}Q_n$ and $\tau(Q_n)=t_n$ for every $n\geqslant 0$.
    By \Cref{prop U*PU-Pj}, there exists a unitary operator $U$ in $\mathcal{M}$ such that
    \begin{equation*}
        \mathcal{M}=W^*(U^*Q_0U,\{Q_n\}_{n=1}^{N}).
    \end{equation*}
    Let $P_0=U^*Q_0U$ and $P_n=Q_n$ for every $n\geqslant 1$.
    Then $\{P_n\}_{n=0}^{N}$ has the desired properties.
\end{proof}

For every real number $x$, the ceiling function $\lceil x\rceil$ is the least integer greater than or equal to $x$.
We provided the best possible integers $m_1$ and $m_2$ in \Cref{lem 1=t1+t2} as follows.

\begin{corollary}\label{cor best-m1-m2}
    Let $(\mathcal{M},\tau)$ be a separable type $\mathrm{II}_1$ factor with $\mathcal{G}(\mathcal{M})=0$.
    Suppose that $0<t\leqslant\frac{1}{2}$ and $m=\max\{2,\lceil t^{-1}\rceil-1\}$.
    Then there are finitely many projections $\{P\}\cup\{Q_j\}_{j=1}^{m}$ in $\mathcal{M}$ such that
    \begin{enumerate} [label= $(\arabic*)$, ref= \arabic*, leftmargin=*] 
        \item[$(1)$] $t=\tau(P)$;
        
        \item[$(2)$] $Q_iQ_j=0$ for $i\ne j$, and $1-t=\sum_{j=1}^{m}\tau(Q_j)$;
        
        \item[$(3)$] $\mathcal{M}=W^*(P,\{Q_j\}_{j=1}^{m})$.
    \end{enumerate}
\end{corollary}

\begin{proof}
    Let $t_0=t$ and $t_j=\frac{1-t}{m}$ for $1\leqslant j\leqslant m$.
    It is clear that
    \begin{equation*}
        t_j\leqslant t_0\leqslant\frac{1}{2}\quad\text{for every}~1\leqslant j\leqslant m.
    \end{equation*}
    We complete the proof by \Cref{cor sum=1}.
\end{proof}

\begin{remark}\label{rem best-m1-m2}
    Note that the constant $m$ is sharp by \Cref{rem U*PU-Pj}.
    Moreover, results similar to \Cref{cor sum=1} and \Cref{cor best-m1-m2} hold for the matrix algebra $M_n(\mathbb{C})$ provided that $t=\frac{k}{n}$ for some positive integer $k$ with $1\leqslant k\leqslant\frac{n}{2}$.
\end{remark}

\section{Similarity of irreducible operators}\label{sec similarity}

In this section, we provide a proof of \Cref{thm similarity}, the main result in this paper.
For every subset $\mathcal{S}$ of $\mathcal{M}$, denote by $\mathcal{S}^c$ its \emph{relative commutant} $\mathcal{S}'\cap\mathcal{M}$.
The \emph{relative bicommutant} of $\mathcal{S}$ is the set $\mathcal{S}^{cc}=(\mathcal{S}^c)^c$ and it is clear that $\mathcal{S}\subseteq\mathcal{S}^{cc}$.
An operator $T$ in $\mathcal{M}$ is said to be \emph{strongly reducible} if $XTX^{-1}$ is reducible for every invertible operator $X$ in $\mathcal{M}$.
For example, in a properly infinite semifinite factor, every operator with finite range projection must be strongly reducible.
For completeness, we sketch a proof of the following lemma, which is essentially proved in \cite[Lemma 2.1]{FJ94}.

\begin{lemma}\label{lem bicommutant}
    Let $\mathcal{M}$ be a separable factor and $T$ a strongly reducible operator in $\mathcal{M}$.
    Then every operator in $\{T\}^{cc}$ is strongly reducible in $\mathcal{M}$.
\end{lemma}

\begin{proof}
    Let $A$ be an operator in $\{T\}^{cc}$ and $X$ an invertible operator in $\mathcal{M}$.
    We claim that $XAX^{-1}$ is reducible.
    By the strong reducibility of $T$, there exists a nontrivial projection $P$ in $\mathcal{M}$ commuting with $XTX^{-1}$.
    Thus, $X^{-1}PX\in\{T\}^c$. This implies that $X^{-1}PX$ commutes with $A\in\{T\}^{cc}$, i.e., $P$ commutes with $XAX^{-1}$.
\end{proof}

Let $T$ and $X$ be operators in $\mathcal{B}(\mathcal{H})$ with $X$ invertible.
It is well-known that $T$ and $XTX^{-1}$ have the same rank and co-rank.
In other words, their range projections $R(T)$ and $R(XTX^{-1})$ are unitarily equivalent.
For the reader's convenience, we sketch a brief proof of this result in general von Neumann algebras, which will be applied in the proof of \Cref{lem main-pre}.

\begin{lemma}\label{lem RP-similarity}
    Let $\mathcal{M}$ be a von Neumann algebra, $T\in\mathcal{M}$, and $X$ an invertible operator in $\mathcal{M}$.
    Then there is a unitary operator $U$ in $\mathcal{M}$ such that $R(XTX^{-1})=U^*R(T)U$.
\end{lemma}

\begin{proof}
    Let $P=R(T)$ and $Q=R(XTX^{-1})$.
    It is routine to verify that $Q=R(XP)$ and $P=R(X^{-1}Q)$.
    By Kaplansky's formula \cite[Theorem 6.1.7]{KR2}, we have
    \begin{equation*}
        P\sim R(T^*)=R(T^*X^*)\sim R(XT)=Q.
    \end{equation*}
    Let $P_0=R(X^*(I-Q))$.
    Then $(I-Q)X(I-P_0)=0$, i.e., $I-P_0=X^{-1}QX(I-P_0)$.
    It is clear that
    \begin{equation*}
        I-P_0=R(X^{-1}QX(I-P_0))\leqslant R(X^{-1}Q)=P.
    \end{equation*}
    Note that $A^*B=0$ if and only if $R(A)R(B)=0$ for any operators $A$ and $B$ in $\mathcal{M}$.
    It follows from $(I-Q)XX^{-1}Q=0$ that $P_0P=0$, i.e., $P\leqslant I-P_0$.
    Thus, we have $P=I-P_0$, and hence
    \begin{equation*}
        I-Q=R((I-Q)X)\sim R(X^*(I-Q))=P_0=I-P.
    \end{equation*}
    Therefore, $P$ and $Q$ are unitarily equivalent in $\mathcal{M}$.
\end{proof}

The ``only if'' part of \Cref{thm similarity} relies on the following lemma.

\begin{lemma}\label{lem main-pre}
    Let $\mathcal{M}$ be a separable factor with $\dim\mathcal{M}>4$.
    Suppose that $T$ is an operator in $\mathcal{M}$ satisfying one of the following conditions:
    \begin{enumerate} [label= $(\arabic*)$, ref= \arabic*, leftmargin=*] 
        \item \label{item-r1} $R(\lambda I-T)\prec I-R(\lambda I-T)$ for some scalar $\lambda$;
        \item \label{item-r2} the set $\{I,T,T^2\}$ is linearly dependent over $\mathbb{C}$.
    \end{enumerate}
    Then $T$ is strongly reducible in $\mathcal{M}$.
\end{lemma}

\begin{proof}
    Suppose that the condition \eqref{item-r1} holds.
    Without loss of generality, we may assume that $\lambda=0$.
    Then $R(T)\prec I-R(T)$.
    Let $S=XTX^{-1}$, where $X$ is an invertible operator in $\mathcal{M}$.
    It is clear that $R(S)\prec I-R(S)$ by applying \Cref{lem RP-similarity}.
    Let $P_1=R(S)$, $P_2=R(S^*)$, and $P=P_1\vee P_2$.
    By Kaplansky's formula, we have
    \begin{equation*}
        P-P_1\sim(P_2-P_1\wedge P_2)\leqslant P_2\sim P_1
        \prec I-P_1.
    \end{equation*}
    It follows that $P\ne I$.
    If $S=0$, then $S$ is clearly reducible.
    If $S\ne 0$, then $S$ is reducible as $P$ is a nontrivial projection commuting with $S$.
    Therefore, $T$ is strongly reducible.
    
    Suppose that the condition \eqref{item-r2} holds.
    Let $S=XTX^{-1}$, where $X$ is an invertible operator in $\mathcal{M}$.
    Then $\{I,S,S^2\}$ is linearly dependent over $\mathbb{C}$.
    In other words, there are scalars $\lambda_1$ and $\lambda_2$ such that $(S-\lambda_1I)(S-\lambda_2I)=0$.
    Let $P_1=R(S-\lambda_2I)$.
    If $P_1=0$, then $S=\lambda_2I$.
    If $P_1=I$, then $S=\lambda_1I$.
    In both cases, $S$ is clearly reducible.
    Thus, we assume that $P_1$ is a nontrivial projection.
    Let $P_2=I-P_1$.
    It is clear that $(S-\lambda_1I)P_1=0$ and $P_1(S-\lambda_2I)=S-\lambda_2I$, i.e., $SP_1=\lambda_1P_1$ and $P_2S=\lambda_2P_2$.
    Bearing in mind the matrix form of $S$ with respect to $I=P_1+P_2$, we obtain that
    \begin{equation*}
        S=(P_1+P_2)S(P_1+P_2)=\lambda_1P_1+P_1SP_2+\lambda_2P_2.
    \end{equation*}
    It is straightforward to verify that both $P_1-R(P_1SP_2)$ and $P_2-R((P_1SP_2)^*)$ are projections in $\mathcal{M}$ commuting with $S$.
    Hence we may assume that $R(P_1SP_2)=P_1$ and $R((P_1SP_2)^*)=P_2$.
    Let $P_1SP_2=V_{12}H$ be the polar decomposition, where $H$ is a positive operator in $P_2\mathcal{M}P_2$ and $V_{12}$ is a partial isometry in $P_1\mathcal{M}P_2$ satisfying that $V_{12}V_{12}^*=P_1$ and $V_{12}^*V_{12}=P_2$.
    By the assumption $\dim\mathcal{M}>4$, there exists a nontrivial projection $Q_2$ in $P_2\mathcal{M}P_2$ commuting with $H$.
    Let $Q=V_{12}Q_2V_{12}^*+Q_2$ be a nontrivial projection in $\mathcal{M}$.
    Note that $S=\lambda_1P_1+V_{12}H+\lambda_2P_2$ and
    \begin{equation*}
    \begin{split}
        QS &=\lambda_1V_{12}Q_2V_{12}^*+V_{12}HQ_2+\lambda_2Q_2,\\
        SQ &=\lambda_1V_{12}Q_2V_{12}^*+V_{12}Q_2H+\lambda_2Q_2.
    \end{split}
    \end{equation*}
    It follows that $S$ commutes with the nontrivial projection $Q$, and hence $S$ is reducible.
    Therefore, $T$ is strongly reducible.
\end{proof}

We now present a technical construction in upper-triangular matrix form that yields irreducible operators.

\begin{lemma}\label{prop upper-triangular}
    Let $\mathcal{M}$ be a separable factor.
    Suppose that
    \begin{enumerate} [label= $(\arabic*)$, ref= \arabic*, leftmargin=*] 
        \item \label{item-ir-1} $\{\alpha_j\}_{j=1}^3$ are distinct scalars and $\{E_j\}_{j=1}^3$ are nonzero projections in $\mathcal{M}$ such that $I=\sum_{j=1}^{3}E_j$;
        \item \label{item-ir-2} $V_{12}\in E_1\mathcal{M}E_2$ and $V_{13}\in E_1\mathcal{M}E_3$ are partial isometries such that
        \begin{equation*}
            V_{12}V_{12}^*=F_1\leqslant E_1,\quad V_{12}^*V_{12}=E_2,\quad V_{13}V_{13}^*=G_1\leqslant E_1,\quad V_{13}^*V_{13}=E_3;
        \end{equation*}
        \item \label{item-ir-3} $A$ is a positive invertible operator in $F_1\mathcal{M}F_1$;
        \item \label{item-ir-4} $B$ is a positive invertible operator in $G_1\mathcal{M}G_1$;
        \item \label{item-ir-5} $W_0^*(A,B)$ is an irreducible subfactor of $E_1\mathcal{M}E_1$.
    \end{enumerate}
    Let $T:=\sum_{j=1}^{3}\alpha_jE_j+AV_{12}+BV_{13}$, which can be expressed in matrix form 
    \begin{equation}  \label{eq-T-matrix}
        T=
        \begin{pmatrix}
            \alpha_1E_1 & AV_{12} & BV_{13} \\
            0 & \alpha_2E_2 & 0 \\
            0 & 0 & \alpha_3E_3
        \end{pmatrix}
        \begin{matrix}
            \operatorname{ran} E_1 \\
            \operatorname{ran} E_2 \\
            \operatorname{ran} E_3
        \end{matrix}.
    \end{equation}
    Then $T$ is an irreducible operator in $\mathcal{M}$.
\end{lemma}

\begin{proof}
Let $Q$ be a projection in $\mathcal{M}$ commuting with $T$.
Since $\{\alpha_j\}_{j=1}^3$ are distinct scalars, $Q$ commutes with $\{E_j\}_{j=1}^3$.
Let $Q_j=QE_j$ for $1\leqslant j\leqslant 3$.
Then
\begin{equation}\label{equ Q1-Q2-Q3}
  Q_1AV_{12}=AV_{12}Q_2\quad\text{and}\quad Q_1BV_{13}=BV_{13}Q_3.
\end{equation}
Let $P=V_{12}Q_2V_{12}^*\in F_1\mathcal{M}F_1$.
Then $Q_1A=AP$, and hence
\begin{equation*}
  Q_1A^2=APA=A(AP)^*=A(Q_1A)^*=A^2Q_1.
\end{equation*}
It follows that $Q_1A=AQ_1$ as $A$ is a positive operator by the condition \eqref{item-ir-3}.
By a similar argument and the condition \eqref{item-ir-4}, $Q_1$ also commutes with $B$.
It follows that $Q_1=0$ or $E_1$ by the condition \eqref{item-ir-5}.
Without loss of generality, we may assume that $Q_1=0$.
Combining with \eqref{equ Q1-Q2-Q3}, we have
\begin{equation*}
  AV_{12}Q_2=0\quad\text{and}\quad BV_{13}Q_3=0.
\end{equation*}
It follows that $Q_2=0$ and $Q_3=0$.
Hence $Q=0$.
Therefore, $T$ is irreducible in $\mathcal{M}$.
This completes the proof.
\end{proof}

Note that $\sum_{j=1}^{3}\alpha_jE_j$ is similar to the operator $T$ given by \eqref{eq-T-matrix}.
Next, we construct operators $V_{12}$, $V_{13}$, $A$, and $B$ as required in \Cref{prop upper-triangular}.
Recall that up to isomorphism, there is a unique separable hyperfinite type $\mathrm{II}_1$ factor, denoted by $\mathcal{R}$.

\begin{lemma}\label{cor upper-triangular}
    Let $\mathcal{M}$ be a separable factor.
    Suppose that $\{\alpha_j\}_{j=1}^3$ are distinct scalars and $\{E_j\}_{j=1}^3$ are nonzero projections in $\mathcal{M}$ such that
    \begin{enumerate} [label= $(\arabic*)$, ref= \arabic*, leftmargin=*, labelsep = 0.5 em] 
        \item $I=\sum_{j=1}^{3}E_j$;
        \item $E_j\precsim I-E_j$ for every $1\leqslant j\leqslant 3$.
    \end{enumerate}
    Then $\sum_{j=1}^{3}\alpha_jE_j$ is similar to an irreducible operator in $\mathcal{M}$.
\end{lemma}

\begin{proof}
    Without loss of generality, we may assume that
    \begin{equation*}
        E_3\precsim E_2\precsim E_1 \quad \text{and} \quad  E_1\precsim E_2+E_3.
    \end{equation*}
    We consider two cases in the proof.
    
    \noindent\textbf{Case I.}
    Suppose that $\mathcal{M}$ is not of type $\mathrm{II}_1$.
    If $\mathcal{M}$ is of type $\mathrm{I}_n$ or type $\mathrm{III}$, then let $\mathcal{A}$ be an arbitrary maximal abelian self-adjoint subalgebra of $E_1\mathcal{M}E_1$.
    If $\mathcal{M}$ is of type $\mathrm{I}_\infty$ or type $\mathrm{II}_\infty$, then there exists a sequence of finite projections $\{P_n\}_{n=1}^{\infty}$ in $E_1\mathcal{M}E_1$ with $E_1=\sum_{n=1}^{\infty}P_n$.
    Clearly, $\{P_n\}_{n=1}^{\infty}$ is contained in a maximal abelian self-adjoint subalgebra $\mathcal{A}$ of $E_1\mathcal{M}E_1$.
    By \eqref{equ masa-In} and \Cref{lem masa-semifinite}, there are projections $F_1$ and $G_1'$ in $\mathcal{A}$ such that
    \begin{equation}\label{equ G1'}
        E_2\sim F_1\leqslant E_1,\quad E_3\sim G_1'\leqslant E_1,\quad\text{and}\quad E_1\leqslant F_1+G_1'.
    \end{equation}
    We claim that there exists a unitary operator $U$ in $E_1\mathcal{M}E_1$ such that
    \begin{equation}\label{equ U-G1'}
        E_1\mathcal{M}E_1=W_0^*(\mathcal{A}F_1,U^*\mathcal{A}G_1'U).
    \end{equation}
    The above claim is obvious in the case that $E_1\mathcal{M}E_1\cong\mathbb{C}$.
    In the following, we assume that $\dim E_1\mathcal{M}E_1>1$.
    If $E_1-F_1=0$, then let $F_0$ be a nonzero projection in $\mathcal{A}G_1'$ such that $F_0\precsim E_1-F_0$.
    If $E_1-F_1\ne 0$, then $E_1-F_1\leqslant G_1'\precsim F_1$ by \eqref{equ G1'} and we take $F_0=E_1-F_1$.
    In both cases, $F_0$ be a nonzero projection in $\mathcal{A}G_1'$ such that
    \begin{equation*}
        E_1-F_1\leqslant F_0\precsim E_1-F_0.
    \end{equation*}
    By \Cref{thm masa}, there exists a unitary operator $U$ in $E_1\mathcal{M}E_1$ such that
    \begin{equation*}
        E_1 \mathcal{M}E_1 = W_0^*(U^*F_0U,\mathcal{A}(E_1-F_0)).
    \end{equation*}
    This proves \eqref{equ U-G1'} as $U^*F_0U\in U^*\mathcal{A}G_1'U$ and $\mathcal{A}(E_1-F_0)\subseteq\mathcal{A}F_1$.
    Let $A$ be a positive invertible operator in $F_1\mathcal{M}F_1$ such that $W_0^*(A)=\mathcal{A}F_1$ and $B$ a positive invertible operator in $G_1\mathcal{M}G_1$ such that $W_0^*(B)=U^*\mathcal{A}G_1'U$, where $G_1=U^*G_1'U$.
    It is clear that $E_1\mathcal{M}E_1=W_0^*(A,B)$.
    Since $E_2\sim F_1$ and $E_3\sim G_1'\sim G_1$, there are partial isometries $V_{12}$ and $V_{13}$ satisfying the condition \eqref{item-ir-2} in \Cref{prop upper-triangular}.
    We complete the proof of \textbf{Case I} by \Cref{prop upper-triangular}.
    
    \noindent\textbf{Case II.}
    Suppose that $\mathcal{M}$ is of type $\mathrm{II}_1$.
    By \cite[Corollary 4.1]{Pop81}, there is an irreducible hyperfinite subfactor $\mathcal{R}$ of $E_1\mathcal{M}E_1$.
    Let $\mathcal{A}$ be an arbitrary maximal abelian self-adjoint subalgebra of $\mathcal{R}$.
    It follows from \eqref{equ masa-II1} that there are projections $F_1$ and $G_1'$ in $\mathcal{A}$ such that
    \begin{equation*}
        E_2\sim F_1\leqslant E_1,\quad E_3\sim G_1'\leqslant E_1,\quad\text{and}\quad E_1\leqslant F_1+G_1'.
    \end{equation*}
    Note that $\mathcal{G}(\mathcal{R})=0$ by \cite[Theorem 3.4]{Shen09}.
    Similar to \eqref{equ U-G1'}, by applying \Cref{thm masa}, there exists a unitary operator $U$ in $\mathcal{R}$ such that
    \begin{equation*}
        \mathcal{R}= W_0^*(\mathcal{A}F_1,U^*\mathcal{A}G_1'U).
    \end{equation*}
    The remaining part follows from the proof of \textbf{Case I}.
\end{proof}

We are now ready to prove the main result of this paper.

\begin{proof}[Proof of \Cref{thm similarity}]
    By \Cref{lem main-pre}, we only need to prove the ``if'' part.
    Let $\mathcal{A}$ be the abelian von Neumann algebra generated by $N$.
    Then $\operatorname{dim}\mathcal{A}>2$ as $\{I,N,N^2\}$ is a linearly independent subset of $\mathcal{A}$.
    If $P$ is a minimal projection in $\mathcal{A}$, then there exists a scalar $\lambda$ such that $P$ is the spectral projection of $N$ with respect to the singleton $\{\lambda\}$, i.e., $I-P=R(\lambda I-N)$.
    It follows that $P\precsim I-P$.
    Therefore, the two conditions on $N$ in \Cref{thm similarity} ensure that $\mathcal{A}$ satisfies the conditions in \Cref{lem P1-P2-P3}.
    
    Let $\{\alpha_j\}_{j=1}^{3}$ be distinct scalars and $\{P_j\}_{j=1}^{3}$ projections in $\mathcal{A}$ as obtained in \Cref{lem P1-P2-P3}.
    Since $N$ is normal, we have $\mathcal{A}^c=\{N\}^c$ by Fuglede's theorem.
    It follows that $\sum_{j=1}^{3}\alpha_jP_j\in\mathcal{A}\subseteq\mathcal{A}^{cc}=\{N\}^{cc}$.
    By \Cref{cor upper-triangular}, $\sum_{j=1}^{3}\alpha_jP_j$ is similar to an irreducible operator in $\mathcal{M}$.
    Therefore, $N$ is similar to an irreducible operator in $\mathcal{M}$ by \Cref{lem bicommutant}.
\end{proof}

Recall that every spectral operator $T$ in $\mathcal{B}(\mathcal{H})$ can be written as $T=S+K$, where the \emph{scalar part} $S$ is similar to a normal operator and the \emph{radical part} $K$ is a quasinilpotent operator commuting with $S$ by \cite{Dun54}.
Similarly, every spectral operator in a von Neumann algebra is also similar to the sum of a normal operator and a commuting quasinilpotent operator by \cite[Section 3]{DK21}.
We present a direct application based on \Cref{thm similarity}.

\begin{corollary}\label{cor spectral-operator}
    Let $\mathcal{M}$ be a separable factor with $\operatorname{dim}\mathcal{M}>4$ and $T$ a spectral operator in $\mathcal{M}$ whose scalar part $S$ satisfies the conditions
    \begin{enumerate}  [label= $(\arabic*)$, ref= \arabic*, leftmargin=*] 
    \item $R(\lambda I-S)\succsim I-R(\lambda I-S)$ for every scalar $\lambda$ in $\mathbb{C}$;
    \item the set $\{I,S,S^2\}$ is linearly independent over $\mathbb{C}$.
    \end{enumerate}
    Then $T$ satisfies the same conditions and is similar to an irreducible operator in $\mathcal{M}$.
\end{corollary}

\begin{proof}
    By \cite[Theorem 3.5]{DK21}, there exists an invertible operator $X$ in $\mathcal{M}$ such that $XSX^{-1}$ is a normal operator.
    There is no loss of generality in assuming that $S$ is normal by considering $XTX^{-1}$.
    Therefore, we can write $T=N+K$, where $N$ is a normal operator in $\mathcal{M}$ and $K$ is a quasinilpotent operator in $\mathcal{M}$ such that $NK=KN$ and the scalar part $N$ satisfies that
    \begin{enumerate}  [label= $(\arabic*)$, ref= \arabic*, leftmargin=*] 
        \item $R(\lambda I-N)\succsim I-R(\lambda I-N)$ for every scalar $\lambda$ in $\mathbb{C}$;
        \item the set $\{I,N,N^2\}$ is linearly independent over $\mathbb{C}$.
    \end{enumerate}
    If $aI+bT+cT^2=0$ for some scalars $a$, $b$, and $c$, then
    \begin{equation*}
        aI+bN+cN^2=-(bI+2cN+cK)K.
    \end{equation*}
    Hence $aI+bN+cN^2$ is normal and quasinilpotent.
    It follows that $aI+bN+cN^2=0$.
    Since $\{I,N,N^2\}$ is linearly independent over $\mathbb{C}$, we have $a=b=c=0$.
    Thus, the set $\{I,T,T^2\}$ is linearly independent.
    For every $n\geqslant 1$, let $E_n$ be the spectral projection of $N$ with respect to $\{z\in\mathbb{C}\colon|z|\geqslant\frac{1}{n}\}$.
    Then $N^{-1}E_n$ is a bounded operator.
    Since $E_n+N^{-1}E_nK$ is an invertible operator in $E_n\mathcal{M}E_n$, we have
    \begin{equation*}
        R(NE_n)=R(NE_n(E_n+N^{-1}E_nK))=R(TE_n)\leqslant R(T).
    \end{equation*}
    Then $R(N)\leqslant R(T)$.
    In general, we have $R(\lambda I-N)\leqslant R(\lambda I-T)$ for every scalar $\lambda$.
    It follows that
    \begin{equation*}
        R(\lambda I-T)\geqslant R(\lambda I-N)\succsim I-R(\lambda I-N)
        \geqslant I-R(\lambda I-T).
    \end{equation*}
    Therefore, the operator $T$ satisfies the two conditions in \Cref{thm similarity}.

    We provide a brief proof of the inclusion $\{T\}^{c}\subseteq\{N\}^{c}$ for completeness, which is essentially proved in \cite[Theorem 5]{Dun54}.
    Let $A\in\{T\}^c$.
    For every Borel subset $\sigma$ of $\mathbb{C}$, denote by $E(\sigma)$ the spectral projection of $N$ with respect to $\sigma$.
    Since $K$ commutes with $N$ and is quasinilpotent, we have
    \begin{equation*}
        \sigma_{E(\sigma)\mathcal{M}E(\sigma)}(TE(\sigma))
        =\sigma_{E(\sigma)\mathcal{M}E(\sigma)}(NE(\sigma))\subseteq\sigma.
    \end{equation*}
    Suppose that $\sigma_1$ and $\sigma_2$ are disjoint closed subsets of $\mathbb{C}$.
    It is clear that
    \begin{equation*}
        \sigma_{E(\sigma_1)\mathcal{M}E(\sigma_1)}(TE(\sigma_1))\cap
        \sigma_{E(\sigma_2)\mathcal{M}E(\sigma_2)}(TE(\sigma_2))=\emptyset.
    \end{equation*}
    Since $T$ commutes with $E(\sigma_j)$ and $A$, we have the following Rosenblum equation
    \begin{equation*}
        TE(\sigma_1)\cdot E(\sigma_1)AE(\sigma_2)
        =E(\sigma_1)AE(\sigma_2)\cdot TE(\sigma_2).
    \end{equation*}
    It follows that $E(\sigma_1)AE(\sigma_2)=0$ by \cite{Ros56} (see also \cite[Lemma 2.4]{MSSW}).
    Thus, we have
    \begin{equation*}
        E(\sigma_1)A(I-E(\sigma_1))=0\quad\text{and}\quad (I-E(\sigma_2))AE(\sigma_2)=0.
    \end{equation*}
    Consequently, $AE(\sigma)=E(\sigma)A$ for every closed subset of $\mathbb{C}$.
    Therefore, $A$ commutes with all spectral projections of $N$, and hence $A\in\{N\}^c$.
    It follows that $\{T\}^{c}\subseteq\{N\}^{c}$.
    Thus, $N\in\{N\}^{cc}\subseteq\{T\}^{cc}$.
    Note that $N$ is similar to an irreducible operator by \Cref{thm similarity}.
    Therefore, we complete the proof by \Cref{lem bicommutant}.
\end{proof}

We end the paper with the following result (see \cite[Proposition 2.18]{JW98-book}).

\begin{proposition}\label{prop commutant-abelian}
    Let $\mathcal{M}$ be a separable factor.
    If $T$ is an operator in $\mathcal{M}$ such that $\{T\}^c$ is abelian, then $T$ is similar to an irreducible operator in $\mathcal{M}$.
\end{proposition}

\begin{proof}
    Since $\{T\}^c$ is abelian, we have $\{T\}^c=\{T\}^{cc}$.
    Let $\mathcal{N}=W^*(T)'\cap\mathcal{M}$.
    Suppose that there exists a minimal projection $P$ in $\mathcal{N}$ such that $P\succsim I-P$.
    Then $TP$ is irreducible in $P\mathcal{M}P$.
    Let $V$ be a partial isometry in $\mathcal{M}$ such that $VV^*=I-P$ and $V^*V\leqslant P$.
    Let $\lambda=\|T\|+1$.
    It is straightforward to verify that $TP+\lambda P$ is an operator in $\{T\}^c=\{T\}^{cc}$ and is similar to $TP+V+\lambda P$, which is irreducible in $\mathcal{M}$ by \cite[Lemma 2.5]{MSSW}.
    We complete the proof by \Cref{lem bicommutant}
    
    Suppose that $P\prec I-P$ for every minimal projection $P$ in $\mathcal{N}$.
    Let $\mathcal{A}$ be a maximal abelian self-adjoint subalgebra of $\mathcal{N}$.
    Since every minimal projection $P$ in $\mathcal{A}$ is a minimal projection in $\mathcal{N}$, we have $P\prec I-P$.
    It follows that $\dim\mathcal{A}>2$.
    Let $\{\alpha_j\}_{j=1}^{3}$ be distinct scalars and $\{P_j\}_{j=1}^{3}$ projections in $\mathcal{A}$ as obtained in \Cref{lem P1-P2-P3}.
    Then $\sum_{j=1}^{3}\alpha_jP_j$ is similar to an irreducible operator by \Cref{cor upper-triangular}.
    Note that
    \begin{equation*}
        \sum_{j=1}^{3}\alpha_jP_j\in\mathcal{A}
        \subseteq\mathcal{N}\subseteq\{T\}^c=\{T\}^{cc}.
    \end{equation*}
    We complete the proof by \Cref{lem bicommutant}.
\end{proof}

\begin{example}\label{eg commutant-abelian}
    It is well known that the commutant algebras of Jordan blocks in $M_n(\mathbb{C})$ and the unilateral shift operator in $\mathcal{B}(\ell^2)$ are abelian.
    By the virtue of the irrational rotation algebra, there exists an operator $T_0$ in the hyperfinite type $\mathrm{II}_1$ factor $\mathcal{R}$ with $\Vert T_0 \Vert \leqslant 2$ 
    such that $\{T_0\}'\cap\mathcal{R}$ is abelian (see the proof of Theorem 8.2 of \cite{FJLX2013}, where we take $T_0 = u+v$).
    Let
    \begin{equation*}
        T:= \bigoplus_{n=1}^{\infty}(2^{-n}I+2^{-n-2}T_0).
    \end{equation*}
    Then $T$ is an operator in the type $\mathrm{II}_\infty$ factor $\mathcal{M}=\mathcal{B}(\ell^2)\, \overline{\otimes}\, \mathcal{R}$ such that $\{T\}'\cap\mathcal{M}$ is abelian.
    When $\mathcal{M}$ is the free group factor $\mathcal{L}_{F_2}$ or a type $\mathrm{III}$ factor, it is interesting to ask whether there exists an operator $T$ in $\mathcal{M}$ such that $\{ T \}' \cap \mathcal{M}$ is abelian.
\end{example}

\section*{Declarations}
\noindent{\bf Conflict of interest}
On behalf of all authors, the corresponding author states that there is no conflict of interest.

\section*{Acknowledgments}
The authors gratefully acknowledge the anonymous referees for their insightful comments and constructive suggestions, as well as research support from Dalian University of Technology during the preparation of this manuscript.


\begin{thebibliography}{99}


\bibitem{CFY21}
X. Cao, J. Fang, Z. Yao.
{\em Strong sums of projections in type  $\mathrm{II}_1$  factors.}
J. Funct. Anal. 281 (2021), no. 5, Paper No. 109088, 11 pp.
MR4253932


\bibitem{CFY24}
X. Cao, J. Fang, Z. Yao.
{\em On finite sums of projections and Dixmier's averaging theorem for type  $\mathrm{II}_1$  factors.}
J. Funct. Anal. 287 (2024), no. 8, Paper No. 110568, 29 pp.
MR4777788


\bibitem{Dun54}
N. Dunford.
\emph{Spectral operators.}
Pacific J. Math. 4 (1954), 321--354.
MR0063563


\bibitem{DK21}
K. Dykema, A. Krishnaswamy-Usha.
\emph{Angles between Haagerup-Schultz projections and spectrality of operators.}
J. Funct. Anal. 281 (2021), no. 4, Paper No. 109027, 26 pp.
MR4249117


\bibitem{DSSW}
K. Dykema, A. Sinclair, R. Smith, S. White.
\emph{Generators of $\mathrm{II}_1$ factors.}
Oper. Matrices 2 (2008), no. 4, 555-582.
MR2468882


\bibitem{DSZ15}
K. Dykema, F. Sukochev, D. Zanin.
\emph{A decomposition theorem in $\mathrm{II}_1$ factors.}
J. Reine Angew. Math. 708 (2015), 97-114.
MR3420330

\bibitem{FJLX2013}
J. Fang, C. Jiang, H. Lin, F. Xu.
\emph{On generalized universal irrational rotation algebras and associated strongly irreducible operators}
Internat. J. Math. 24 (2013), no. 8, 1350059, 35 pp.
MR3103875

\bibitem{FT69}
P. Fillmore, D. Topping.
\emph{Sums of irreducible operators.}
Proc. Amer. Math. Soc. 20 (1969), 131--133.
MR0233226


\bibitem{FJ94}
C. Fong, C. Jiang.
\emph{Normal operators similar to irreducible operators.}
Acta Math. Sinica (N.S.) 10 (1994), no. 2, 132--135.
MR1398037


\bibitem{Ge03}
L. Ge.
\emph{On ``Problems on von Neumann algebras by R. Kadison, 1967''.}
Acta Math. Sin. (Engl. Ser.) 19 (2003), no. 3, 619--624.
MR2014042


\bibitem{Gil72}
F. Gilfeather.
\emph{Strong reducibility of operators.}
Indiana Univ. Math. J. 22 (1972/73), 393--397.
MR0303322


\bibitem{HS09}
U. Haagerup, H. Schultz.
\emph{Invariant subspaces for operators in a general $\mathrm{II}_1$ factor.}
Publ. Math. Inst. Hautes \'Etudes Sci. No. 109 (2009), 19--111.
MR2511586


\bibitem{HMS26}
D. Hadwin, M. Ma, J. Shen.
{\em Voiculescu's theorem in properly infinite factors.}
J. Funct. Anal. 290 (2026), no. 1, Paper No. 111198, 34 pp.
MR4964138


\bibitem{Hal68}
P. Halmos.
\emph{Irreducible operators.}
Michigan Math. J. 15 (1968), 215--223.
MR0231233


\bibitem{Hsin00}
C. Hsin.
\emph{Finite matrices similar to irreducible ones.}
Taiwanese J. Math. 4 (2000), no. 3, 457--477.
MR1779110


\bibitem{JW98-book}
C. Jiang, Z. Wang.
\emph{Strongly irreducible operators on Hilbert space.}
Pitman Res. Notes Math. Ser., 389,
Longman, Harlow, 1998. x+243 pp.
MR1640067


\bibitem{JW06-book}
C. Jiang, Z. Wang.
\emph{Structure of Hilbert space operators.}
World Scientific Publishing Co. Pte. Ltd., Hackensack, NJ, 2006. x+248 pp.
MR2221863


\bibitem{KR1}
R. Kadison, J. Ringrose.
\emph{Fundamentals of the theory of operator algebras, I, Elementary theory.}
Academic Press, New York, 1983.
MR0719020


\bibitem{KR2}
R. Kadison, J. Ringrose.
\emph{Fundamentals of the theory of operator algebras, II, Advanced theory.}
Academic Press, Orlando, FL, 1986.
MR0859186


\bibitem{MSSW}
M. Ma, J. Shen, R. Shi, T. Wang.
\emph{Products of irreducible operators in factors.}
arXiv:2512.12162, 2025


\bibitem{Pop81}
S. Popa.
\emph{On a problem of R. V. Kadison on maximal abelian $*$-subalgebras in factors.}
Invent. Math. 65 (1981/82), no. 2, 269--281.
MR0641131


\bibitem{Ros56}
M. Rosenblum.
\emph{On the operator equation $BX-XA=Q$.}
Duke Math. J. 23 (1956), 263--269.
MR0079235


\bibitem{Shen09}
J. Shen.
\emph{Type $\mathrm{II}_1$ factors with a single generator.}
J. Operator Theory 62 (2009), no. 2, 421--438.
MR2552089


\bibitem{SS19}
J. Shen, R. Shi.
{\em Reducible operators in non-$\Gamma$ type $\mathrm{II}_1$ factors.}
\emph{Math. Ann.} 394 (2026), no. 2, Paper No. 31, 38 pp.


\bibitem{SS-book}
A. Sinclair, R. Smith.
\emph{Finite von Neumann algebras and masas.}
London Math. Soc. Lecture Note Ser., 351 Cambridge University Press, Cambridge, 2008. x+400 pp.
MR2433341


\bibitem{SS63}
N. Suzuki, T. Saito.
\emph{On the operators which generate continuous von Neumann algebras.}
Tohoku Math. J. (2) 15 (1963), 277--280.
MR0154143


\bibitem{Wog69}
W. Wogen.
\emph{On generators for von Neumann algebras.}
Bull. Amer. Math. Soc. 75 (1969), 95--99.
MR0236725


\end{thebibliography}
\end{document}